\documentclass[oneside, headsepline, 10pt, titlepage=false]{scrartcl}

\usepackage[intlimits]{amsmath}
\usepackage{amsfonts,amssymb,amsxtra}
\usepackage{verbatim} %

\usepackage[left,mathlines,pagewise]{lineno}
\newcommand*\patchAmsMathEnvironmentForLineno[1]{%
\expandafter\let\csname old#1\expandafter\endcsname\csname #1\endcsname
\expandafter\let\csname oldend#1\expandafter\endcsname\csname end#1\endcsname
\renewenvironment{#1}%
{\linenomath\csname old#1\endcsname}%
{\csname oldend#1\endcsname\endlinenomath}}%
\newcommand*\patchBothAmsMathEnvironmentsForLineno[1]{%
\patchAmsMathEnvironmentForLineno{#1}%
\patchAmsMathEnvironmentForLineno{#1*}}%
\AtBeginDocument{%
\patchBothAmsMathEnvironmentsForLineno{equation}%
\patchBothAmsMathEnvironmentsForLineno{align}%
\patchBothAmsMathEnvironmentsForLineno{flalign}%
\patchBothAmsMathEnvironmentsForLineno{alignat}%
\patchBothAmsMathEnvironmentsForLineno{gather}%
\patchBothAmsMathEnvironmentsForLineno{multline}%
}

\usepackage{array} 
\newenvironment{ma}{\begin{array}{>{\displaystyle}r >{\displaystyle}c >{\displaystyle}l}}{\end{array}}%

\usepackage{amsmath,amsxtra,amsthm}
\usepackage[varg]{txfonts}
\usepackage[unicode]{hyperref}
\usepackage{framed}
\usepackage{mathrsfs}
\usepackage{graphicx}

\makeatletter
\def\th@plain{%
\thm@headfont{\bfseries}
\thm@notefont{}
  \itshape 
}
\makeatother

\newtheorem{theorem}{Theorem}%

\newtheorem{lemma}{Lemma}[section]%
\newtheorem{proposition}[lemma]{Proposition}%
\newtheorem{corollary}[lemma]{Corollary}%

\newcounter{mythmct}
\hypersetup{
breaklinks=true,
colorlinks=true,
linkcolor=blue,
citecolor=blue,
urlcolor=blue,
}

\allowdisplaybreaks

\def\bfseries{\fontseries \bfdefault \selectfont \boldmath}

\makeatletter
\def\@fnsymbol#1{\ensuremath{\ifcase#1\or *\or **\or {**}* \or {**}{**}\else\@ctrerr\fi}}
\makeatother

\newcommand{\aleq}{\prec}
\renewcommand{\prec}{\lesssim}
\newcommand{\ageq}{\succ}
\renewcommand{\succ}{\gtrsim}
\newcommand{\aeq}{\approx}

\renewcommand{\Upsilon}{Y}

\newcommand{\N}{{\mathbb N}}
\newcommand{\R}{{\mathbb R}}

\newcommand{\Z}{{\mathbb Z}}
\renewcommand{\S}{{\mathbb S}}

\newcommand{\fracm}[1]{{\frac{1}{#1}}}

\newcommand{\supp}{\operatorname{supp}}%
\newcommand{\dist}{\operatorname{dist}}%

\newcommand{\lap}{\Delta}

\newcommand{\Dsymbol}{\abs{D}}
\newcommand{\Rieszsymbol}{I}

\newcommand{\laps}[1]{\Dsymbol^{#1}}
\newcommand{\lapms}[1]{\Rieszsymbol_{#1}}

\newcommand{\lapv}{\Dsymbol^{\frac{1}{2}}}
\newcommand{\lapmv}{\Rieszsymbol_{\frac{1}{2}}}
\newcommand{\lapsv}[1]{\Dsymbol^{\frac{#1}{2}}}
\newcommand{\lapmsv}[1]{\Rieszsymbol_{\frac{#1}{2}}}

\newcommand{\abs}[1]{{\left \vert #1 \right \vert}}%
\newcommand{\sabs}[1]{{\vert #1 \vert}}%
\newcommand{\brac}[1]{{\left ( #1 \right )}}%
\newcommand{\sbrac}[1]{{\langle #1 \rangle}}%
\newcommand{\Sbrac}[1]{{\left \langle #1 \right \rangle}}%
\newcommand{\Vrac}[1]{{\left \Vert #1 \right \Vert }}%
\newcommand{\vrac}[1]{{ \Vert #1 \Vert }}%

\newcommand{\ontop}[2]{{\genfrac{}{}{0pt}{}{#1}{#2}}}

\newcommand{\intl}{\int \limits}

\def\XXint#1#2#3{{\setbox0=\hbox{$#1{#2#3}{\int}$}%
     \vcenter{\hbox{$#2#3$}}\kern-.5\wd0}}%

\newcommand{\eps}[1][]{\varepsilon_{#1}}

\newcommand{\br}[1]{\ensuremath{\left(#1\right)}}

\renewcommand{\d}{\ensuremath{\,\mathrm{d}}}

\newcommand{\norm}[1]{\ensuremath{\left\|#1\right\|}}

\newcommand{\seminorm}[1]{\ensuremath{\left[#1\right]}}
\newcommand{\seq}[1]{\ensuremath{\br{#1}_{\eps>0}}}

\newcommand{\set}[1]{\ensuremath{\left\{#1\right\}}}

\newcommand{\sett}[2]{\ensuremath{\left\{#1\left|\,#2\right.\right\}}}
\renewcommand{\sp}[1]{\ensuremath{\left\langle#1\right\rangle}}

\numberwithin{equation}{section}

\newcommand{\capH}{\ensuremath{\cap H^{\scriptstyle1,\infty}}}

\newcommand{\D}[1]{{#1}(u+w)-{#1}(u)}

\newcommand{\dg}{{\g}'}
\newcommand{\dgn}{{\g_0'}}
\newcommand{\dgt}{\gt'}
\renewcommand{\dh}{{\varphi'}}

\newcommand{\dtg}{{\tilde{\g}}'}

\newcommand{\E}[1][]{\ensuremath{E_{#1}^{(2)}}}
\newcommand{\Ee}[1][\eps]{\ensuremath{E_{#1}^{(2)}}}
\renewcommand{\eps}{\ensuremath{\varepsilon}}

\newcommand{\g}{\gamma}
\newcommand{\gt}[1][\tau]{\g_{#1}}
\renewcommand{\G}[1][\g]{\ensuremath{\mathcal G_{#1}}}
\newcommand{\h}{\ensuremath{\varphi}}

\DeclareMathOperator{\id}{id}

\newcommand{\length}{\ensuremath{\mathscr L}}
\DeclareMathOperator{\lip}{lip}

\renewcommand{\N}{\ensuremath{\mathbb{N}}}

\renewcommand{\R}{\ensuremath{\mathbb{R}}}

\newcommand{\refeq}[2][=]{\ensuremath{\stackrel{\text{\makebox[0cm][c]{\eqref{eq:#2}}}}{#1}}}
\renewcommand{\rho}{\ensuremath{\varrho}}
\newcommand{\rzd}{\ensuremath{(\R/\Z,\R^n)}}
\newcommand{\s}{\ensuremath{\sigma}}

\newcommand{\stackrl}[2][=]{\ensuremath{\stackrel{\text{\makebox[0cm][c]{#2}}}{#1}}}

\newcommand{\tdg}{{\tilde\gamma'}}
\newcommand{\tg}{\tilde{\g}}

\newcommand{\W}[1][3/2]{\ensuremath{H^{\scriptstyle #1}}}

\newcommand{\Wir}[1][3/2]{\ensuremath{H_{\mathrm{ir}}^{#1}}}
\newcommand{\Hir}[1][3/2]{\ensuremath{H_{\mathrm{ir}}^{#1}}}

\numberwithin{equation}{section}%

\title{Hard analysis meets critical knots \\ \Large Stationary points of the M\"obius energy are smooth}
\author{S. Blatt\thanks{Simon Blatt, Mathematics Institute, Zeeman Building, University of Warwick, Coventry CV4 7AL, United Kingdom, \url{S.Blatt@warwick.ac.uk}}
\and$\Phi$. Reiter\thanks{Philipp Reiter, Abteilung f\"ur Angewandte Mathematik, Universit\"at Freiburg, Hermann-Herder-Stra\ss e 10, 79104 Freiburg i.~Br., Germany, \url{reiter@mathematik.uni-freiburg.de}}
\and A. Schikorra\thanks{Armin Schikorra, Max-Planck Institute for Mathematics in the Sciences, Inselstr. 22, 04103 Leipzig, Germany, \url{armin.schikorra@mis.mpg.de}}}

\date{\today}

 \usepackage{fancyhdr}

\begin{document}
\maketitle


\begin{abstract}
 We prove that if a curve  $\g \in H^{3/2}(\mathbb R / \mathbb Z, \mathbb R^n)$
 parametrized by arc length is
 a stationary point of the M\"obius energy introduced by Jun O'Hara in \cite{OH91}, then $\g$ is 
 smooth.
 Our methods only rely on purely analytical arguments, entirely without using
 M\"obius invariance. Furthermore, they are not fundamentally
 restricted to one-dimensional domains, but are generalizable to arbitrary dimensions. 
\end{abstract}

\tableofcontents

\newcommand{\soln}{\gamma}
\newcommand{\solnt}{g}
\newcommand{\dsoln}{\solnt'}
\newcommand{\dsolnt}{\solnt'}
\newcommand{\soltest}{\phi}
\newcommand{\tsoltest}{{{\tilde{\phi}}}}

\newcommand{\tdim}{{n}}

\renewcommand{\g}{{\soln}}
\newcommand{\gn}{{\gamma_0}}

\section{Introduction}
To find nice geometric representatives within a given knot class, several new energies have been invented 
in the last two decades. 
The earliest knot energy for smooth curves  was the so-called \emph{M\"obius energy} introduced 
by Jun O'Hara
in \cite{OH91}, 
\[
 E^{(2)}(\soln) := \intl_{\R/\Z}\intl_{-1/{2}}^{1/{2}} 
  \brac{\fracm{\abs{\soln(u+w)-\soln(u)}^2}-\fracm{d_{\soln}(u+w,u)^2}}\ \abs{\soln'(u+w)}\abs{\soln'(u)}\ 
 \d w\d u
\]
which was later on extended to the family of energies
\[
E^{(\alpha,p)}(\soln) := \intl_{\R/\Z}\intl_{-1/{2}}^{1/{2}} 
  \brac{\fracm{\abs{\soln(u+w)-\soln(u)}^\alpha}-\fracm{d_{\soln}(u+w,u)^\alpha}}^p\ \abs{\soln'(u+w)}\abs{\soln'(u)}\ 
 \d w\d u
\]
for $\alpha, p \in [1,\infty)$, see~\cite{OHara1994}.
Here $d_\g(u+w,u)$ denotes the intrinsic distance between $\g(u+w)$ and $\g(u)$ on the curve~$\g$.
More precisely,
\begin{equation}\label{eq:intrinsic}
 d_\g(u+w,u):= \min\big(\length(\g|_{[u,u+w]}),\length(\g)-\length(\g|_{[u,u+w]})\big)
\end{equation}
provided $\abs w\le\tfrac12$ where
$
 \length(\g) := \int_0^1 \abs{\dg(\theta)}\d\theta
$
is the length of~$\g$.

Crucially using the M\"obius invariance of this knot energy, Michael Freedman, Zheng-Xu He, and Zhenghan Wang \cite{Freedman1994} were able to show that there are minimizers of the M\"obius energy within every 
prime knot class and that these are in fact of class $C^{1,1}$. More precisely, they
could show that if $\g$ is a local minimizer with respect to the $L^\infty$-topology, and 
if $\g$ is parametrized by arc length, then $\g$ is $C^{1,1}$.
Together with a bootstrapping argument due to He \cite{He2000}, one then obtains that local 
minimizers of the M\"obius energy are smooth, also see~\cite{Reiter2009a}.

Unfortunately, motivated by numerical evidence, Rob Kusner and John Sullivan were led to conjecture 
that there are no minimizers within composite knot classes \cite{Kusner1997}. In contrast to that, there
are minimizers of the energies $E^{(\alpha,p)}$ in the case that $jp>2$ as shown in \cite{OHara1994}. 

In this article, we will prove that even only stationary points of the M\"obius energy are of class
$C^\infty$ under the mildest condition one can think of: that $E^{(2)}(\g)$
is finite - an assumption which, as shown in a recent work of the first author \cite{Blatt2010a}, is equivalent to 
assuming that $\g$ is an injective curve of class $H^{\frac 3 2}(\R / \Z, \R ^n)$.
Our motivation to do so is twofold:
First of all, of course, 
this is a much stronger result than the smoothness of 
local minimizers as stated above.
Secondly, the M\"obius invariance is
not essential for proving smoothness of local minimizers as we do not need it in our arguments.
Thus there is the chance
to study other, possibly not M\"obius invariant, critical knot energies, using the techniques
developed in this article.
Additionally, our arguments are not restricted to the
one-dimensional situation but can be applied to arbitrary dimensions.
 
The price we pay is that, instead of the very appealing geometric argument in \cite{Freedman1994}, we have to adapt some 
sophisticated techniques originally developed  by Tristan Rivi\`ere and Francesca Da Lio \cite{DR11Sphere,DR11Man,DL11MannD}
and the third author \cite{SchikorranharmSphere,SchikorraIntBoundFrac11} to deal with $\frac n2$- harmonic maps into manifolds.

The first task in order to prove this result, is to derive the Euler-Lagrange equation for such stationary
points. In \cite{Freedman1994}, it was shown that for simple closed curves 
$\g \in C^{1,1}(\mathbb R / \mathbb Z, \mathbb R^n)$ and $h\in C^{1,1}(\mathbb R / \mathbb Z, \mathbb R^n)$
we have
\begin{align*}
 \delta E ^{(2)}(\g;h) &:= \lim_{\tau \searrow 0 } \frac {E^{(2)}(\g+\tau h) - E^{(2)}(\g)}\tau \\
 &= \int_{\mathbb R / \mathbb Z} \int_{\mathbb R / \mathbb Z} 
  \left(\frac {\g'(u) h'(u) - \frac {\left\langle \g(u)- \g(v), h(u)-h(v) \right\rangle}
      {|\g(u)-\g(v)|^2}}{|\g(u)-\g(v)|^2 }\right)\ |\g'(v)|\ |\g'(u)|\ dv\ du.
\end{align*} 
 
We will show that this formula is still valid under the weaker assumption that 
$\g \in H^{3/2}(\R / \Z,$ $\R^\tdim)$ with $\dg\in L^{\infty}$, i.e., for
arc-length parametrized $\g$ we only assume that the M\"obius energy
is finite.
We call a curve \emph{regular}
if there is a positive constant $c=c(\g)$ with
$|\g'(x)| \geq c$ {for all } $x \in \mathbb R/ \mathbb Z$.

\begin{theorem}[$E^{(2)}\in C^1(\Hir\capH)$] \label{thm:EulerLagrangeEquation}
 The energy $\E$ is continuously differentiable on the space of injective and regular
 curves belonging to $\W\capH$. Furthermore, if
 $\g\in\W$ is injective and parametrized by arc-length
 and $\h\in \W\capH$
   the first variation \[ \delta\E(\g;\h) := \lim_{\tau\to0}\frac{\E(\g+\tau\h)-\E(\g)}\tau \]
   exists and equals
   \[
     2\lim_{\eps\searrow0} \iint\limits_{U_\eps}
     \br{\sp{\dg(u),\dh(u)}-\frac{\sp{\D\g,\D\h}}{\abs{\D\g}^2}} 
     \frac{\d w\d u}{\abs{\D\g}^2}
   \]
   where
   \begin{equation}\label{eq:Ue}
     U_\eps := \R/\Z\times\br{[-\tfrac12,-\eps]\cup[\eps,\tfrac12]}.
   \end{equation}
\end{theorem}

Though the space $\W\capH$ seems somewhat artificial at first sight, it just guarantees that 
we do not use bad parametrizations of our curves. The proof of this 
result is similar to~\cite{BlattReiterStationary11}.

We will then use the resulting Euler-Langrange equation for stationary points of the M\"obius energy to
prove that these points are smooth:

\begin{theorem}[Stationary points are smooth] \label{thm:MainResult}
Any stationary point $\soln \in H^{\frac 3 2}(\R / \Z, \R^n)$ of $E^{(2)}$, i.e., any curve
$\soln \in H^{\frac 3 2 }(\R / \Z, \R^n)$ for which
\begin{equation*}
 \delta E^{(2)} (\g;h) =0 \quad \mbox{for all $h \in C^{\infty}(\mathbb R / \mathbb Z, \mathbb R^n)$},
\end{equation*}
belongs to $C^\infty$ when parametrized by arc-length.
\end{theorem}

In \cite{BlattReiterStationary11}, improving a previous result~\cite{ReiterOharaKnot11},
the smoothness of stationary points of finite energy was already shown
for the case of $E^{(\alpha)}:= E^{(\alpha,1)}$, $\alpha \in (2,3)$, instead of the M\"obius energy.
It is worth noting, that those energies lead to 
a \emph{subcritical} Euler-Langrage equation, and that in some sense the regularity theory can be based on  Sobolev 
embeddings for fractional Sobolev and Besov spaces. 
In contrast to this, the Euler-Lagrange equation of the M\"obius energy is \emph{critical}. As for 
well-known critical geometric equations -- like the Euler-Lagrange equation of the
Willmore functional, see, e.g., \cite{Simon1993, Riviere2006},
or harmonic maps on $\mathbb R^2$, see, e.g., \cite{Helein1991,Riviere2007}, -- one has first to find a way 
to gain an $\eps$ of additional regularity (via gaining a $\delta$ of additional integrability) and then start a bootstrapping argument.
That is, in a quite natural way, the proof of Theorem~\ref{thm:MainResult} is an immediate consequence of two technically independent steps:

\begin{theorem}[Initial regularity]\label{th:reg:1alpha} \label{thm:FirstStep}
Let $\soln \in H^{\frac{3}{2}}(\R /\Z,\R^\tdim)$, $\soln' \in \S^{\tdim-1}$, be a stationary point of the M\"obius energy, i.e. satisfying $\delta\E(\g,\varphi) = 0$ for all $\varphi\in C^\infty\rzd$.
Then $\soln \in C^{1,\alpha}$, for some $\alpha > 0$, and  $\soln \in H^{\frac{3}{2},p}(\R /\Z,\R^\tdim)$ for some $p > 2$.
\end{theorem}

\begin{theorem}[Bootstrapping] \label{thm:bootstrapping}
 For some $p > 2$, let $\soln \in H^{\frac{3}{2},p}(\R /\Z,\R^\tdim)$, $\soln' \in \S^{\tdim-1}$,
 be a stationary point of the M\"obius energy. Then $\soln$ is smooth.
\end{theorem}
Theorem~\ref{th:reg:1alpha} is proven in Section \ref{sec:initial}, Theorem~\ref{thm:bootstrapping} in Section \ref{sec:bootstrapping}.
While Theorem~\ref{thm:bootstrapping} relies mainly on bringing together Sobolev embeddings and 
standard commutator estimates for Bessel potential spaces with techniques developed in
\cite{Blatt2010a}, some very delicate estimates
are needed to get anything more than the critical and initial regularity $H^{\frac 3 2}$ for 
stationary points as stated in Theorem~\ref{thm:FirstStep}.

Both theorems rely on a decomposition of the first variation dating back to~\cite[Formula~(4.5)]{He2000} which 
already proved to be helpful in the analysis of the functionals 
$E^{(\alpha)}$ for $\alpha \in (2,3)$ (cf. \cite{BlattReiterStationary11}) and the gradient flow 
of the energies $E^{(\alpha)}$ for $\alpha \in [2,3)$ \cite{Blatt2011,Blatt2011c}. 

For  $f, g: \R \to \R^\tdim$, $\eps>0$ let 
\begin{equation}\label{eq:defQeps}
 Q_\varepsilon(f,g) := 
  \intl_{0}^1 \intl_{[-\frac{1}{2},\frac{1}{2}]\backslash (-\varepsilon,\varepsilon)} \br{\sbrac{{f}'(u), {g}'(u)}\ w^2 - \Sbrac{ f(u+w)-f(u), g(u+w)-g(u)}}
  \frac{dw}{w^4}\ du,
\end{equation}
\begin{equation}\label{eq:defT1}
 T_1(f,g) := - \intl_{0}^{1}\intl_{-\frac{1}{2}}^{\frac{1}{2}}\sbrac{{f}'(u),\ {g}'(u) }\ \brac{\fracm{\abs{f(u+w)-f(u)}^2}-\fracm{\abs{w}^2}}\ dw\ du,
\end{equation}
and
\begin{equation}\label{eq:defT2}
 T_2(f,g) := \intl_0^1\intl_{-\frac{1}{2}}^{\frac{1}{2}} \sbrac{ f(u+w)-f(u),\ g(u+w)-g(u) }\ \brac{\fracm{\abs{f(u+w)-f(u)}^4}-\fracm{\abs{w}^4}}\ dw\ du.
\end{equation}
From Theorem~\ref{thm:EulerLagrangeEquation} we deduce
that a critical knot $\soln \in H^{\frac{3}{2}}(\R /\Z,\R^\tdim)$,  parametrized by arc length, satisfies
\begin{equation}
\label{eq:knotPDE}
 Q(\soln,h) := \lim_{\varepsilon \to 0} Q_\varepsilon(\soln,h) = T_1(\soln,h) + T_2(\soln,h) \quad \mbox{for all $h \in C^\infty(\R /\Z)$}.
\end{equation}
This is the form of the Euler-Lagrange equation,  with which we will work in the proofs of Theorem~\ref{thm:FirstStep}
and Theorem~\ref{thm:bootstrapping}.

Let us conclude this introduction by remarking that, in contrast to stationary
points of $E^{(\alpha,1)}$, for $p>1$ we do not expect stationary points of $E^{(\alpha,p)}$ to be $C^\infty$-smooth:
The resulting Euler-Lagrange equation should be in some sense
a nonlocal degenerate elliptic equation. 
Keeping in mind the regularity theory for elliptic degenerate equations, one might expect nevertheless
that stationary points are at least 
a bit more regular than an arbitrary finite-energy curve alone.

\paragraph*{Acknowledgements.}
The first author was supported by Swiss National Science Foundation Grant Nr.~200020\_125127 and
the Leverhulm trust.
The second author was supported by DFG Transregional Collaborative Research Centre SFB~TR~71.
The third author has received funding from the European Research Council under the European Union's Seventh Framework Programme (FP7/2007-2013) / ERC grant agreement no 267087, DAAD PostDoc Program (D/10/50763) and the Forschungsinstitut f\"ur Mathematik, ETH Z\"urich. He would like to thank Tristan Rivi\`ere and the ETH for their hospitality.

\section{Euler-Lagrange equation: Proof of Theorem \ref{thm:EulerLagrangeEquation}}

This section is devoted to the proof of Theorem~\ref{thm:EulerLagrangeEquation}
which especially involves the derivation of a formula for the first variation.

By $\Wir$ we will denote the set of \underline injective and \underline regular curves in~$\W$.
The set $\Hir[1,\infty]$ is defined accordingly.
First we will need
the following lemma, to guarantee that $E^{(2)}$ is well defined on a
sufficiently small $\W\capH$ neighborhood of the curve $\g$:


\begin{lemma}[$\Hir \capH$ is open in $\W \capH$]\label{lem:Upsilon}
 For any $\g\in\Hir \rzd \capH \rzd$ there is some $\tau_0=\tau_0(\g)>0$ with
 \begin{equation}\label{eq:Upsilon}
  \Upsilon := \sett{\g+\h}{\h\in H^{1,\infty}\rzd, \norm\dh_{L^\infty}\le\tau_0}\subset\Hir[1,\infty].
 \end{equation}
 Moreover, there is a constant $c=c(\g)>0$ with
 \begin{equation} \label{eq:fundamentalEstimates}
 \min\{|\tg(u+w) - \tg(u)|, d_{\tg}(u+w,u)\}  \geq c |w|, \quad \quad
 \quad|\dtg(u)| \geq c
 \end{equation}
 for all $\tg \in \Upsilon$ and $(u,w) \in U_0$.
\end{lemma}

\begin{proof}
  We first show that $\g$ is bi-Lipschitz. To this end, choose $\delta\in(0,\frac12)$ with
\begin{equation*}
 \left(\int_{B_r(z)} \int_{B_r(0)} \frac
   {|\dg(u+w)-\dg(u)|^{2}}{|w|^{2}} \d w \d u\right)^{1/2} \leq \tfrac12
\end{equation*}
for all $z \in \mathbb R / \mathbb Z$ and all $r \in[0,\delta]$ which gives
\begin{align*}
  &\frac 1 {2r}\int_{B_r (z)}\bigg| \dg(x) - \frac 1 {2r}  \int_{B_r(z)} \dg(y) \d y\bigg| \d x \\
  &\leq \frac  1 {4r^2} \int_{B_r(z)} \int_{B_r(z)}  |\dg(x) - \dg(y)| \d x \d y \\
  &\leq \left(  \frac 1 {4r^2} \int_{B_r(z)} \int_{B_r(z)}
  |\dg(x) - \dg(y)|^{2} \d x \d y\right)^{1/2} \\
  &\leq \left( \int_{B_r(z)}  \int_{B_r(z)}
  \frac{|\dg(x) - \dg(y)|^{2}}{|x-y|^{2}} \d x \d y\right)^{1/2} \\
  &\leq \tfrac 1 2 .
\end{align*}
Since $\abs{\frac 1 {2r}  \int_{B_r(z)} \dg(y)
  \d y} \leq 1$  we deduce that
\begin{equation*}
  \inf_{\substack{a\in \mathbb R^n \\ |a|\leq 1}} \frac 1 {2r} \int_{B_r(z)} \abs{\dg(y)-a} \d y \leq \tfrac 1 2.
\end{equation*}
For $x,y \in \mathbb R / \mathbb Z$ with $|x-y| \leq 2\delta$ let $r:=
\tfrac12|x-y|$ and $z \in \mathbb R / \mathbb Z$ be the midpoint of
the shorter arc between $x$ and $y$. Then
\begin{align*}
  |\g(x)-\g(y)| &= \sup_{\substack{a\in \mathbb R^n \\ |a|\leq 1}}\;
  \int_{B_r(z)}\langle \dg(t),a\rangle \d t\\
 &= \sup_{\substack{a\in \mathbb R^n \\ |a|\leq 1}} \; \int_{B_r(z)}\langle
 \dg(t) ,\dg(t) + (a-\dg(t))\rangle \d t\\
&\geq \left( 1 - \inf_{\substack{a\in \mathbb R^n \\ |a|\leq 1}}
  \frac 1 {2r}\int_{B_{r}(z) }|\dg(t)-a|\d t \right) |x-y| \\
  &\geq\tfrac 12 |x-y|
\end{align*}
for all $x,y \in \mathbb R / \mathbb Z$ with $|x-y|\leq 2\delta$.
Since $\g$ is embedded and
\[ (x,y) \mapsto \frac
{|\g(y)-\g(x)|}{|y-x|} \]
 defines a continuous positive function on
$I_\delta :=\{(x,y) \in (\mathbb R / \mathbb Z)^2: |x-y| \geq 2\delta\}$, we furthermore have
\begin{equation*}
  |\g(x) - \g(y)| \geq \underbrace{\min_{(\tilde x,\tilde y) \in I_\delta}\frac {|\g(\tilde y)-\g(\tilde x)|}{|\tilde y-\tilde x|}}_{>0} |x-y|.
\end{equation*}
for all $(x,y) \in I_\delta$.
Hence, there is a $c_0=c_0(\g)>0$ with
\begin{equation*}
 |\g(x) - \g(x+w)| \geq c_0 |w|
\end{equation*}
for all $w \in [-1/2, 1/2]$.
Lessening $c_0$ if necessary, we can also achieve by regularity
\[ \abs{\dg}\ge c_0 \qquad\text{on }\R/\Z. \]

Letting $\tau_0:=\tfrac12 c_0$ we obtain
for arbitrary $\tg\in \Upsilon$
\begin{align*}
 \abs{\tg(u+w)-\tg(u)} &\ge \abs{\g(u+w)-\g(u)} - \abs{(\tg-\g)(u+w)-(\tg-\g)(u)} \\
 &\ge c_0\abs w-\norm{\dg-\tdg}_{L^\infty}\abs w \\
 &\ge \tfrac12c_0\abs w
\end{align*}
and
\[ \abs{\tdg}\ge\abs{\dg}-\abs{\tdg-\dg}\ge c_0-\norm{\dg-\tdg}_{L^\infty}\ge\tfrac12c_0. \]
From the latter estimate
we deduce by~\eqref{eq:intrinsic} for $u\in\R/\Z$, $w\in[-\tfrac12,\tfrac12]$
\[ d_{\tg}(u+w,u) \ge d_{\tg}(u\pm w,u) \ge \tfrac12c_0\abs w. \]
We have established~\eqref{eq:fundamentalEstimates} which gives~\eqref{eq:Upsilon}.
\end{proof}

We will use the last lemma to prove the following theorem, from which 
Theorem~\ref{thm:EulerLagrangeEquation} will follow quite easily.

\begin{proposition}\label{prop:formula}
  The energy $\E$ is continuously differentiable on $\br{\Wir\capH}\rzd$.
  The derivative of $\E$ at $\g \in {\Wir\capH}$ in direction $\h \in {\Wir\capH}$ exists and
  is given by
  \begin{equation}\label{eq:formulaD}
  \begin{split}
    &\delta\E(\g;\h)  = \\
    &2\lim_{\varepsilon \searrow 0} \iint\limits_{U_\eps} \Bigg\{
    \left(
      \frac 1 {\abs{\D\g}^2}-\frac 1 {d_\g(u+w,u)^2} \right)
      \left\langle \frac{\dg(u)}{\abs{\dg(u)}^2},\dh(u) \right\rangle 
\\
    &\qquad\qquad {}-
      \left( \frac{\sp{\D\g,\D\h}}{\abs{\D\g}^{4}}
	-\frac{ \left.\tfrac{\d}{\d\tau}\right|_{\tau=0}d_{\g+\tau\h}(u+w,u)}
		{d_\g(u+w,u)^{3}} \right)\Bigg\} 
\\ &\hspace{7.6cm} |\dg(u+w)| |\dg(u)|\d w\d u .
  \end{split}
  \end{equation}
\end{proposition}

As $\g$ is absolutely continuous and regular, the derivative 
$\left.\tfrac{\d}{\d\tau}\right|_{\tau=0}d_{\g+\tau\h}(u+w,u)$ is well-defined for almost all
$(u,w) \in U_0$. From~\eqref{eq:intrinsic} we deduce
\begin{multline}\label{eq:dd} 
 \left.\tfrac{\d}{\d\tau}\right|_{\tau=0} d_{\g+\tau\h} (u+w,u) =
 \begin{cases} |w|\int_{0}^1 
    \left\langle \frac {\dg(u+\sigma w)}{|\dg(u+\sigma w)|} , \dh(u+\sigma w) \right\rangle d\sigma
    &\text{if } \length(\g|_{[u,u+w]}) <\frac 1 2 \length(\g),\\
    -|w|\int_{0}^1
    \left\langle  \frac {\dg(u+\sigma w)}{|\dg(u+\sigma w)|} , \dh(u+\sigma w) \right\rangle d\sigma
    &\text{if } \length(\g|_{[u,u+w]}) > \frac 1 2 \length(\g).
    \end{cases}
\end{multline}

To prove Proposition~\ref{prop:formula}, we will first show that the following approximations
of the energy $\E$, in which we cut off the singular part, are continuously differentiable
and provide a formula for the first variation.
For $\varepsilon\in(0,\tfrac12)$ we set
\begin{equation*}
  \Ee(\g) := \iint\limits_{ U_\eps} 
    \br{\frac1{\abs{\D\g}^2}-\frac1{d_\g(u+w,u)^2}}\abs{\dg(u+w)}\abs{\dg(u)}
    \d w\d u.  
\end{equation*}

\begin{lemma} \label{lem:FrechetDifferentiable}
For $\eps\in(0,\tfrac12)$ the functional $\Ee$ is continuously differentiable on the space of all
injective regular curves in $\W \rzd \cap H^{1,\infty}\rzd$.
The directional derivative at $\g\in \Wir \rzd \cap H^{1,\infty} \rzd$ in direction 
$\h\in \W \rzd \cap H^{1,\infty} \rzd$ is given by
\begin{equation}\label{eq:FrechetDifferentiable}
\begin{split}
    &\delta\Ee(\g;\h)  = \\
    & 2\iint\limits_{U_\eps} \Bigg\{
    \left(
      \frac 1 {\abs{\D\g}^2}-\frac 1 {d_\g(u+w,u)^2} \right)
      \left\langle \frac{\dg(u)}{\abs{\dg(u)}^2},\dh(u) \right\rangle 
\\
    &\qquad \qquad {}-
      \left( \frac{\sp{\D\g,\D\h}}{\abs{\D\g}^{4}}
	-\frac{ \left.\tfrac{\d}{\d\tau}\right|_{\tau=0}d_{\g+\tau\h}(u+w,u)}
		{d_\g(u+w,u)^{3}} \right)\Bigg\} 
\\ &\hspace{7.6cm} |\dg(u+w)| |\dg(u)|\d w\d u .
  \end{split}
\end{equation}
\end{lemma}

\begin{proof}
Applying Lemma~\ref{lem:Upsilon}, we obtain an $\W \capH$-neighborhood $Y \subset \Wir \capH$ of $\g$
such that~\eqref{eq:fundamentalEstimates} uniformly holds on $Y$ for any element in $U_\varepsilon$.
The integrand in~\eqref{eq:FrechetDifferentiable} is almost everywhere
the pointwise derivative of the integrand in $\Ee$. Using~\eqref{eq:intrinsic} and~\eqref{eq:dd}, one sees
furthermore that this pointwise derivative is majorized by some $L^1$-function. So,
Lebesgue's Theorem permits to interchange differentiation and integration which,
by a suitable reparametrization, results in~\eqref{eq:FrechetDifferentiable}.

As for continuity of $\Ee$ and $\delta\Ee$ the only difficulty is to treat the intrinsic distance.
Recalling the continuity of the length functional with respect to absolutely continuous curves
we can directly read off from~\eqref{eq:intrinsic} that
the integrand of $\Ee$ defines a continuous operator
$(H_{\mathrm{ir}}^{3/2}\cap H^{1,\infty})(\mathbb R / \mathbb Z, \mathbb R^n) \to L^1(U_{\varepsilon})$.

Since $\gamma$ is regular, for any $u\in\R/\Z$ there are at most two points $w\in[-\tfrac12,\tfrac12]$ satisfying
$\mathscr L(\gamma|_{[u,u+w]}) = \tfrac 1 2 \mathscr L(\gamma)$ which results in a null set in $U_0$.
Additionally using~\eqref{eq:dd}, we see that the integrand of $\delta\Ee$ gives rise to a continuous mapping
$(H_{\mathrm{ir}}^{3/2}\cap H^{1,\infty})(\mathbb R / \mathbb Z, \mathbb R^n) \times H^{1,\infty} (\mathbb R / \mathbb Z, \mathbb R^n)
\to L^1(U_{\varepsilon})$.
Being linear and bounded in the second component, it can be viewed as a continuous mapping from $\W[1,\infty]$
into the linear bounded operators $\Hir\capH\to L^1(U_\eps)$.

Altogether, the integrand of $\Ee$ is a continuously differentiable functional 
$H_{\mathrm{ir}}^{3/2}\capH\to L^1(U_\eps)$.
The statement now follows from the chain rule and the fact that the integration operator
\[ L^1(U_\varepsilon) \to \mathbb R, \qquad
 g \mapsto \iint_{U_\varepsilon} g(u,w) \d u\d w, \]
is continuously differentiable as it is a bounded linear operator.
\end{proof}

Due to the fact that bounded $L^1$-sequences are not uniformly integrable, the approximations $\Ee$ do not even form
a Cauchy sequence in $C^0(\Hir\capH)$.
In order to prove Proposition~\ref{prop:formula}, we state in Lemma~\ref{lem:CauchySequence} below
that $\Ee$ is \emph{nearly} a Cauchy sequence 
in $C^1(X_{\delta})$ for subsets $X_\delta \subset \Wir\capH$, $\delta\ge0$,
which satisfy the following substitute of the uniform integrability property
\begin{equation}\label{eq:uniformintegrability}
 \limsup_{\varepsilon \rightarrow 0}\sup_{\gamma \in X_\delta} 
  \left(\iint_{\mathbb R / \mathbb Z \times [-\eps, \eps]}
  \frac {|\gamma'(u+w) - \gamma'(u)|^2}{w^2} \d w \d u \right) ^{1/2}
\leq \delta.
\end{equation}
The statement involves the Lipschitz constant
\[ \lip_Y E = \sup_{\substack{f,\tilde f\in Y\\f\ne\tilde f}}\frac{\abs{E(f)-E(\tilde f)}}{\norm{f-\tilde f}} \]
for some real-valued functional $E$ and a subset $Y$ contained in its domain.

\begin{lemma} \label{lem:CauchySequence}
 We have
 \begin{equation*}
  \Ee(\gamma) \rightarrow \E(\gamma)
 \end{equation*}
 for all $\gamma \in \Hir \capH$.

 Furthermore, for any $\gn \in \Wir\capH$ there is an open subset $\Upsilon \subset \Wir\capH$ 
 and a constant $C=C(\gn) < \infty$ such that
\begin{equation}
 \limsup_{\varepsilon_1, \varepsilon_2 \rightarrow 0}\lip_{X_\delta\cap\Upsilon} (\E[\eps_1] - \E[\eps_2])
 \le C \delta\label{eq:CauchySequence2}
\end{equation}
for all subsets $X_\delta \subset \W\capH$ satisfying \eqref{eq:uniformintegrability} with $\delta \in[0,1]$.
\end{lemma}

\begin{proof}
 From Lemma~\ref{lem:Upsilon} we get a $\W \capH$-neighborhood
 $\Upsilon \subset \Wir \capH$ of $\gamma$ such that
 \eqref{eq:fundamentalEstimates} holds for all $\g\in\Upsilon$.
 Making $\Upsilon$ smaller if
 necessary, we may also assume the existence of an $\varepsilon_0 >0$ with
 \begin{equation*}
  d_{\g}(u+w,u) = \mathscr L (\g|_{[u,u+w]})
 \end{equation*}
 for all $\g \in \Upsilon$ and  $w \in [-\varepsilon_0, \varepsilon_0]$.

 \renewcommand{\gt}{\tg}
 \renewcommand{\dgt}{\tg'} 
 In order to bring the integrand in the definition of $\E$ and $\Ee$
 in a more convenient form we
 introduce the function
 \[ g(\zeta,\eta,\vartheta,\iota) := \frac{\zeta^{-2}-\eta^{-2}}{\eta^2-\zeta^2}\vartheta\iota \]
 which is Lipschitz continuous and positive on $[\tilde c,\infty)^4$ for any $\tilde c>0$.
 We define for $u\in\R/\Z$, $w\in[-\eps_0,\eps_0]$
 \begin{align*}
  \G[\g]:(u,w) &\mapsto g\br{\abs{\int_0^1\dg(u+\theta_1w)\d\theta_1},\int_0^1\abs{\dg(u+\theta_2w)}\d\theta_2,\abs{\dg(u+w)},\abs{\dg(u)}}.
 \end{align*}
 We have chosen $\Upsilon$ in such a way that the arguments in $\G[]$ are 
 uniformly bounded away from zero.
 Then we decompose the integrand in the definition of $\E$ for $|w| \leq \varepsilon_0$ into
 \begin{align*}
  &\br{\frac1{\abs{\D\g}^2}-\frac1{d_{\g}(u+w,u)^2}}\abs{\dg(u+w)}\abs{\dg(u)}\\
  &=\frac1{w^2}\br{\frac1{\abs{\int_0^1\dg(u+\theta_1w)\d\theta_1}^2}-\frac1{\br{\int_0^1\abs{\dg(u+\theta_2w)}\d\theta_2}^2}}\abs{\dg(u+w)}\abs{\dg(u)}\\
  &=\G[\g](u,w)\frac{\br{\int_0^1 \abs{\dg(u+\theta_2w)}\d\theta_2}^2 - \abs{\int_0^1\dg(u+\theta_1w)\d\theta_1}^2}{\abs w^{2}}  \\
  &=\G[\g](u,w)\frac{\iint_{[0,1]^2} \br{\abs{\dg(u+\theta_1w)}\abs{\dg(u+\theta_2w)} - \sp{\dg(u+\theta_1w),\dg(u+\theta_2w)}} \d\theta_1\d\theta_2}{\abs w^{2}}.
 \end{align*}
 Using $2\abs a\abs b-2\sp{a,b} = \abs{a-b}^2 - \abs{\abs a-\abs b}^2$ for $a,b\in\R^n$ 
 this can be written as
 \begin{align*}
  &\tfrac12 \G[\g](u,w)\frac{\iint_{[0,1]^2}\abs{\dg(u+\theta_1w)-\dg(u+\theta_2w)}^2\d\theta_1\d\theta_2}{w^2} \\
  &\quad{}-\tfrac12 \G[\g](u,w)\frac{\iint_{[0,1]^2}\br{\abs{\dg(u+\theta_1w)}-\abs{\dg(u+\theta_2w)}}^2\d\theta_1\d\theta_2}{w^2}.
 \end{align*}
 We first use this to get
 \begin{align*}
  \E(\gamma) - \Ee(\gamma) \leq C \int_{\mathbb R / \mathbb Z} \int_{-\varepsilon}^\varepsilon
  \G[\g](u,w)\frac{\iint_{[0,1]^2}\abs{\dg(u+\theta_1w)-\dg(u+\theta_2w)}^2\d\theta_1\d\theta_2}{w^2} 
  dw du  \\
  \leq C  \int_{\mathbb R / \mathbb Z} \int_{-\varepsilon}^\varepsilon
  \frac{\iint_{[0,1]^2}\abs{\dg(u+\theta_1w)-\dg(u+\theta_2w)}^2\d\theta_1\d\theta_2}{w^2} 
  dw du 
  \xrightarrow{\varepsilon \rightarrow 0} 0 
 \end{align*}
 which proves the pointwise convergence stated in the lemma.

 Let now $0<\varepsilon_1 < \varepsilon_2 < \varepsilon_0$ and set
 \[ F:= \E[\varepsilon_1] - \E[\varepsilon_2]. \]
 Using the decomposition of the integrand above, we get 
 \begin{align*}
  F(\g) &=  \tfrac12 \int_{\mathbb R / \mathbb Z}
  \int_{\varepsilon_1 < |w| < \varepsilon_2}\G[\g](u,w)\frac{\iint_{[0,1]^2}\abs{\dg(u+\theta_1w)-\dg(u+\theta_2w)}^2\d\theta_1\d\theta_2}{w^2}\d w\d u \\
  &\quad{}-\tfrac12  \int_{\mathbb R / \mathbb Z}
  \int_{\varepsilon_1 < |w| < \varepsilon_2}\G[\g](u,w)\frac{\iint_{[0,1]^2}\br{\abs{\dg(u+\theta_1w)}-\abs{\dg(u+\theta_2w)}}^2\d\theta_1\d\theta_2}{w^2}\d w\d u \\
  &=: \tfrac12 F_1^{(2)}(\g) - \tfrac12 F_2^{(2)}(\g).
 \end{align*}
 To estimate the difference $F(\gt)-F(\g)$ for $\g,\tg\in Y_{0}$, we first consider
 \begin{align*}
  &\abs{\G[\gt](u,w)-\G(u,w)} \\
  &\le C\abs{\abs{\int_0^1\dtg(u+\theta_1w)\d\theta_1}-\abs{\int_0^1\dg(u+\theta_2w)\d\theta_2}}
  \\
  & \quad{}+C\abs{\int_0^1\br{\abs{\dgt(u+\theta w)}-\abs{\dg(u+\theta w)}}\d\theta} \\
  &\quad{}+C\abs{\abs{\dgt(u+w)}-\abs{\dg(u+w)}} +C\abs{\abs{\dgt(u)}-\abs{\dg(u)}} \\
  &\le C\int_0^1\abs{\dgt(u+\theta w)-\dg(u+\theta w)}\d\theta +C\abs{{\dgt(u+w)}-{\dg(u+w)}} +C\abs{{\dgt(u)}-{\dg(u)}} \\
  &\le C\norm{\dgt-\dg}_{L^\infty}.
 \end{align*}
 We arrive at
 \begin{align*}
  &\abs{F_1^{(2)}(\gt)-F_1^{(2)}(\g)} \\
  &\le \iint\limits_{\R/\Z\times[-\eps_2,\eps_2]}\abs{\G[\gt](u,w)-\G(u,w)}\frac{\iint_{[0,1]^2}\abs{\dgt(u+\theta_1w)-\dgt(u+\theta_2w)}^2\d\theta_1\d\theta_2}{w^2}\d w\d u \\
  &\quad{}+\iint\limits_{\R/\Z\times[-\eps_2,\eps_2]}\abs{\G[\g](u,w)}\\
  & \qquad \qquad  \frac{\iint_{[0,1]^2}\abs{\abs{\dgt(u+\theta_1w)-\dgt(u+\theta_2w)}^2-\abs{\dg(u+\theta_1w)-\dg(u+\theta_2w)}^2}\d\theta_1\d\theta_2}{w^2} \\
  & \hspace{11.2cm}\d w\d u \\
  &\le C\norm{\dgt-\dg}_{L^\infty}\iint\limits_{[0,1]^2}\iint\limits_{\R/\Z\times[-\eps_2,\eps_2]}\frac{\abs{\dgt(u+\theta_1w)-\dgt(u+\theta_2w)}^2}{w^2}\d w\d u\d\theta_1\d\theta_2 \\
  &\quad{}+C\norm\dg_{L^{\infty}}\iint\limits_{[0,1]^2}\iint\limits_{\R/\Z\times[-\eps_2,\eps_2]} \\
  & \qquad  \frac{\abs{(\dgt+\dg)(u+\theta_1w)-(\dgt+\dg)(u+\theta_2w)}\abs{(\dgt-\dg)(u+\theta_1w)-(\dgt-\dg)(u+\theta_2w)}}{w^2} \\
  & \hspace{10.2cm}\d w\d u\d\theta_1\d\theta_2 \\
  &\le C\seminorm{\dgt}_{\W[1/2]_{\scriptstyle 2\eps_2}}^2\norm{\dgt-\dg}_{L^\infty}
    + C \norm\dg_{L^{\infty}}\seminorm{\dgt+\dg}_{\W[1/2]_{\scriptstyle 2\eps_2}}\seminorm{\dgt-\dg}_{\W[1/2]_{\scriptstyle 2\eps_2}}
 \end{align*}
 where we set for $\eps\in(0,\tfrac12)$
 \begin{equation*}
  \seminorm{f}_{\W[1/2]_{\scriptstyle\eps}}
  := \left( \iint_{\R/\Z\times[-\eps,\eps]} \frac{\abs{f(u+w)-f(u)}^2}{w^{2}} \d w\d u\right)^{1/2}.
 \end{equation*}
 For the second term we compute
 \begin{align*}
  &\abs{F_2^{(2)}(\gt)-F_2^{(2)}(\g)} \\
  &\le \! \! \iint\limits_{\R/\Z\times[-\eps_2,\eps_2]} \! \! \! \!\abs{\G[\gt](u,w)-\G(u,w)}\frac{\iint_{[0,1]^2} \!\! \br{\abs{\dgt(u+\theta_1w)}-\abs{\dgt(u+\theta_2w)}}^2\d\theta_1\d\theta_2}{w^2}\d w\d u \\
  &\qquad{}+\iint\limits_{\R/\Z\times[-\eps_2,\eps_2]}\abs{\G[\g](u,w)} \cdot\\
  &\qquad \cdot\frac{\iint_{[0,1]^2}\abs{\br{\abs{\dgt(u+\theta_1w)}-\abs{\dgt(u+\theta_2w)}}^2-\br{\abs{\dg(u+\theta_1w)}-\abs{\dg(u+\theta_2w)}}^2}\d\theta_1\d\theta_2}{w^2}\\
  & \hspace{11.2cm} \d w\d u \\
  &\le C\norm{\dgt-\dg}_{L^\infty}\iint\limits_{[0,1]^2}\iint\limits_{\R/\Z\times[-\eps_2,\eps_2]}\frac{\abs{\dgt(u+\theta_1w)-\dgt(u+\theta_2w)}^2}{w^2}\d w\d u\d\theta_1\d\theta_2 \\
  &\quad{}+C\norm\dg_{L^{\infty}}\iint\limits_{[0,1]^2}\iint\limits_{\R/\Z\times[-\eps_2,\eps_2]}\frac{\abs{\br{\abs{\dgt(u+\theta_1w)}-\abs{\dgt(u+\theta_2w)}}+\br{\abs{\dg(u+\theta_1w)}-\abs{\dg(u+\theta_2w)}}}}{\abs w}\cdot \\
  &\qquad\quad{}\cdot\frac{\abs{\br{\abs{\dgt(u+\theta_1w)}-\abs{\dgt(u+\theta_2w)}}-\br{\abs{\dg(u+\theta_1w)}-\abs{\dg(u+\theta_2w)}}}}{\abs w}\d w\d u\d\theta_1\d\theta_2 \\
  &\le C\norm{\dgt-\dg}_{L^\infty}\seminorm{\dgt}_{\W[1/2]_{\scriptstyle 2\eps_2}}^2
      + C\norm\dg_{L^{\infty}}\seminorm{\abs\dgt+\abs\dg}_{\W[1/2]_{\scriptstyle 2\eps_2}}\seminorm{\abs\dgt-\abs\dg}_{\W[1/2]_{\scriptstyle 2\eps_2}}.
 \end{align*}
 Using the chain and product rule for Sobolev spaces and the formula
 \[
  \abs\dgt-\abs\dg 
  = \frac{\sp{\dgt+\dg,\dgt-\dg}}{\abs\dgt+\abs\dg},
 \]
 we obtain, assuming $\eps_{0}<\tfrac14$, for $C>0$ depending on
 the constant from~\eqref{eq:fundamentalEstimates}
 \begin{align*}
  \seminorm{\abs\dgt-\abs\dg}_{\W[1/2]_{\scriptstyle 2\eps_2}} 
  &\le \seminorm{\abs\dgt-\abs\dg}_{\W[1/2](\R/\Z)} \\
  &\le C \norm{\dgt-\dg}_{\W[1/2]\cap L^{\infty}}
  \norm{\dgt+\dg}_{\W[1/2]\cap L^{\infty}}\br{1+\seminorm\tdg_{\W[1/2]_{\scriptstyle 2\eps_2}}+\seminorm\dg_{\W[1/2]_{\scriptstyle 2\eps_2}}}
  \end{align*}
 where
 \[ \norm\cdot_{\W[1/2]\cap L^{\infty}} := \norm\cdot_{\W[1/2]}+\norm\cdot_{L^{\infty}}. \]
 Hence
 \begin{align*}
    \abs{F_2^{(2)}(\gt)-F_2^{(2)}(\g)}
    &\le C \Bigg(\seminorm{\dgt}_{\W[1/2]_{\scriptstyle 2\eps_2}}^2 +
    \norm\dg_{L^{\infty}}\br{\seminorm\tdg_{\W[1/2]_{\scriptstyle 2\eps_2}}+\seminorm\dg_{\W[1/2]_{\scriptstyle 2\eps_2}}}
    \norm{\dgt+\dg}_{\W[1/2]\cap L^{\infty}}\cdot \\
    &\qquad\qquad\qquad\qquad{}\cdot\br{1+\seminorm\tdg_{\W[1/2]_{\scriptstyle 2\eps_2}}+\seminorm\dg_{\W[1/2]_{\scriptstyle 2\eps_2}}}
     \Bigg)
    \|\dgt-\dg\|_{\W[1/2]\cap L^{\infty}}.
 \end{align*}
 The claim follows from
 \[ \limsup_{\eps_{2}\searrow0}\sup_{\tg\in X_{\delta}\cap Y}\seminorm{\dtg}_{\W[1/2]_{\scriptstyle2\eps_{2}}} \le \delta. \]
\end{proof}

\begin{proof}[Proof of Proposition~\ref{prop:formula}]
 In order to prove that \emph{directional derivatives} exist at $\gn\in\Hir\capH$ for all directions $\h\in\W\capH$
 let $\Upsilon$ be as in the proof of Lemma~\ref{lem:CauchySequence} and
 \begin{equation*}
  X_0 := \sett{ \gn + \tau\h}{\tau \in (-1,1)}.
 \end{equation*}
 First we observe that $X_0$ satisfies~\eqref{eq:uniformintegrability}, thus being an admissible set
 for Lemma~\ref{lem:CauchySequence}. Indeed,
 for $\gt := \gn +\tau\h$, $|\tau| \leq 1$,
 \begin{align}
  &\Bigg( \iint_{\mathbb R / \mathbb Z \times [-\eps, \eps]} 
  \frac {|\g_\tau'(u+w) - \g_\tau'(u)|^2}{|w|^2} \d w\d u \Bigg)^{1/2}  \nonumber \\
  &\leq \left( \iint_{\mathbb R / \mathbb Z \times [-\eps, \eps]}
  \frac {|\dg_{0}(u+w) - \dg_{0}(u)|^2}{|w|^2} \d w \d u \right)^{1/2}
  +  \left( \iint_{\mathbb R / \mathbb Z \times [-\eps, \eps]}
  \frac {|\dh(u+w) - \dh(u)|^2}{|w|^2} \d w \d u \right)^{1/2} 
  \nonumber\\
  &\xrightarrow{\eps\searrow0} 0.\label{eq:X0}
 \end{align}
 From this we deduce,
 for
 \[ f_\eps : \tau\mapsto \Ee(\g_{0}+\tau\h), \]
 that
 \begin{align}
  &\abs{f_{\eps_1}'(\tau)-f_{\eps_2}'(\tau)}
  =\abs{\delta\Ee[\eps_1](\g_0+\tau \h;\h)-\delta\Ee[\eps_2](\g_0+\tau \h;\h)} \nonumber\\
  &\le\limsup_{\theta\to0}\abs{\frac{\Ee[\eps_1](\g_0+(\tau+\theta)\h)-\Ee[\eps_1](\g_0+\tau \h)}{\theta}
   - \frac{\Ee[\eps_2](\g_0+(\tau+\theta)\h)-\Ee[\eps_2](\g_0+\tau \h)}{\theta}} \nonumber\\  
  &\le \lip_{X_0\cap\Upsilon}\br{\Ee[\eps_1]-\Ee[\eps_2]}\norm \h_{\W\cap\W[1	,\infty]}\label{eq:f-f}\\
  &\xrightarrow{\eps_1,\eps_2\searrow0}0\qquad\text{by \eqref{eq:CauchySequence2}.}\nonumber
 \end{align}
 As $\Ee\to\E$ pointwise, this proves that $\seq{f_\eps}$
 is a Cauchy sequence in $C^1((-\tau_0, \tau_0))$ converging to $\E(\g_0 + \tau \h)=\lim_{\eps\searrow0}\Ee(\g_0+\tau \h)$ as 
 $ \varepsilon \rightarrow 0$.
 Hence, especially all directional derivatives of $\E$ exist and
 \begin{equation*}
  \delta \E (\g_0;\h) = \lim_{\varepsilon \searrow 0} \delta \Ee (\g_0;\h)
 \end{equation*}
 for all $\g_0 \in \Wir\cap\W[1,\infty]$, $\h \in \W\cap\W[1,\infty]$.
 
  The next step is to establish \emph{G\^ateaux differentiability}. To this end we merely have to show
 $\delta\E(\g_0,\cdot)\in\br{\W\cap\W[1,\infty]}^*$ for $\g_0\in\Wir\cap\W[1,\infty]$. Linearity carries over from $\Ee$.
 In order to prove boundedness we introduce
 \begin{equation}\label{eq:X_delta}
  X_\delta := \{ \g \in \W : \|\g - \gn\|_{\W} \leq \delta\}\qquad\text{for }\delta\in(0,1]
 \end{equation}
 which also satisfies~\eqref{eq:uniformintegrability} as
 for $\g \in X_\delta$ we have
 \begin{equation}\label{eq:Xdelta}
 \begin{split}
  &\Bigg( \iint_{\mathbb R / \mathbb Z \times [-\eps, \eps]}
  \frac {|\dg(u+w) - \dg(u)|^2}{|w|^2} \d w \d u \Bigg)^{1/2} \\
  &\leq \left( \iint_{\mathbb R / \mathbb Z \times [-\eps, \eps]}
  \frac {|\dgn(u+w) - \dgn(u)|^2}{|w|^2} \d w \d u \right)^{1/2}
 +  \delta \xrightarrow{\eps\searrow0} \delta.
 \end{split}  
 \end{equation}
 Now 
 \begin{align*}
  \delta \E(\gn;\h) 
  &= \delta \Ee(\gn;\h) + \delta \E(\gn;\h) - \delta \Ee(\gn;\h)
  \\
  &= \delta \Ee(\gn;\h) + \lim_{\varepsilon_1 \searrow0} 
   \br{\delta \E[\varepsilon_1] (\gn;\h) - \delta \Ee(\gn;\h)}
 \end{align*}
 and thus, arguing as in~\eqref{eq:f-f} and recalling $\delta\Ee(\g_0;\cdot)\in\br{\W\capH}^*$,
\begin{equation*}
\begin{aligned}
  | \delta \E(\gn;\h) |& \leq | \delta \Ee(\gn;\h) |
  +  \underbrace{\limsup_{\eps_1\searrow0}\lip_{X_\delta\cap\Upsilon}\br{\Ee[\eps_1]-\Ee[\eps]}}_{<\infty} \|\h\|_{\W\cap\W[1,\infty]}
  \leq C \|\h\|_{\W\cap\W[1,\infty]}
\end{aligned}
\end{equation*}
 for all $\gn \in \Wir \capH$ and $\h \in \W \capH$. Hence, $\E$ is G\^ateaux differentiable  and the 
differential is given by
\begin{equation*}
 \left(\E\right)'(\gn) = \delta \E(\gn;\cdot)
\end{equation*}
for all $\gn \in \Wir \capH$, $\h \in \W \capH$.

Finally, to see that the differential is \emph{continuous}, let $\sigma >0$ be given and let us choose
$\delta>0$ and $\varepsilon>0$ so small that
\begin{equation*}
 \lip_{X_\delta\cap\Upsilon}\br{\Ee[\eps_1]-\Ee[\eps_2]} \refeq[\le]{CauchySequence2} C \delta \leq \tfrac13 \sigma
\end{equation*}
for all $\varepsilon_1, \varepsilon_2 < \eps$ where $X_\delta$ is as in~\eqref{eq:X_delta}. Then we
have for $\g \in X_\delta \cap\Upsilon$ and any $\h\in\W\capH$
\begin{align*}
 |\delta E^{(2)}(\g;\h) - \delta E^{(2)}(\gn;\h)|
 &\leq |\delta \E(\g;\h) - \delta \Ee(\g;\h)| + 
  |\delta \Ee(\g;\h) - \delta \Ee(\gn;\h)| 
 \\ 
  &\quad + |\delta \Ee(\gn;\h) - \delta \E(\gn;\h)|
 \\
  &\refeq[\le]{f-f}  |\delta \Ee(\g;\h) - \delta \Ee(\gn;\h)| +  \tfrac23 \sigma \norm\h_{\W\capH}.
\end{align*} 
Since $\Ee$ is $C^1$ we deduce that there is an open neighborhood $V\subset X_\delta$
of $\gn$ such that 
\begin{equation*}
 |\delta \Ee(\g;\h) - \delta \Ee(\gn;\h)| \leq \tfrac13\sigma \norm\h_{\W\capH}
\end{equation*}
and hence
\begin{equation*}
  |\delta \E(\g;\h) - \delta \E(\gn;\h)| \leq \sigma \norm\h_{\W\capH}
\end{equation*}
for all $\g\in V$.
This proves that $\left(\E \right)'$ is continuous from $\Wir\capH$ into $\left( \W \capH \right)^\ast$ and
hence $\E$ is $C^1(\Wir$ $\capH)$.
\end{proof}

\begin{proof}[Proof of Theorem~\ref{thm:EulerLagrangeEquation}]
 Using Proposition~\ref{prop:formula} we merely have to derive the formula of the first variation
 for a curve $\g \in \Wir\capH$ parametrized by arc-length
 and $\h \in \W\capH$.
 As $\abs\dg\equiv1$ a.~e.\@ we deduce from Lemma~\ref{lem:FrechetDifferentiable} 
 and Equation~\eqref{eq:dd}
 \begin{align*}
  \delta\E(\g;\h) &\xleftarrow{\varepsilon \searrow 0} 2\iint_{U_\varepsilon}
\Bigg\{\left(\frac 1 {|\g(u+w) - \g(u)|^2} - \frac 1 {w^2} \right)
  \left\langle \dg(u), \dh(u) \right\rangle  \\
  &\qquad {}- 2  \left(\frac{\left\langle \D\g ,\D \h \right\rangle }{|\D\g|^{4}}
    -\frac{\left.\tfrac{\d}{\d\tau}\right|_{\tau=0}d_{\g+\tau\h}(u+w,u)}{|w|^{3}}\right)\Bigg\}
  \d w\d u 
 \end{align*}
 where now 
 \begin{equation*}
  \left.\tfrac{\d}{\d\tau}\right|_{\tau=0}d_{\g+\tau\h}(u+w,u) = |w|\int_0^1 \left\langle \dg(u+ \theta w), \dh(u+\theta w) \right\rangle d\theta
 \end{equation*}
 for all $(u,w) \in \mathbb R / \mathbb Z \times (-\tfrac12, \tfrac12)$. Hence,
 \begin{align*}
  &\delta\E(\g;\h) 
  \xleftarrow{\varepsilon \searrow 0}2\iint_{U_\varepsilon}  \Bigg\{
    \left(\frac 1 {|\g(u+w) - \g(u)|^2} - \frac 1 {w^2} \right) 
    \left\langle \dg(u), \dh(u) \right\rangle \\
    &\quad{} -   \left(\frac{\left\langle \D\g ,\D \h \right\rangle }{|\D\g|^{4}}
    -\frac{\int_0^1 \left\langle \dg(u+\theta w ), \dh(u+\theta w) \right\rangle d\theta}
    {w^2}\right) \Bigg\}\d w\d u \\
  & = 2\iint_{U_\varepsilon} \Bigg\{
    \left(\frac 1 {|\g(u+w) - \g(u)|^2} - \frac 1 {w^2} \right) 
    \left\langle \dg(u), \dh(u) \right\rangle \\
    &\qquad\qquad{} - \left(\frac{\left\langle \D\g ,\D \h \right\rangle }{|\D\g|^{4}}
    -\frac{\left\langle \dg(u), \dh(u) \right\rangle}{w^2}\right) \Bigg\}\d w\d u\\
& = 2\iint_{U_\varepsilon} 
  \br{\frac  {\left\langle \dg(u), \dh(u) \right\rangle} {|\g(u+w) - \g(u)|^2}
            - \frac{\left\langle \D\g ,\D \h \right\rangle }{|\D\g|^{4}} } \d w\d u.
\end{align*}
\end{proof}

\section{Initial regularity: Proof of Theorem \ref{th:reg:1alpha}}\label{sec:initial}
Note that if we
consider the constant factors to be irrelevant with respect to the mathematical argument, for the sake of simplicity
we will omit them in the calculations, writing $\aleq , \ageq , \aeq$ instead of $\leq, \geq$ and $=$.

Most techniques for dealing with critical partial differential equations 
of fractional order have been developed for equations on the whole Euclidean space. For that reason, we
prefer working on the real line over working on the circle.

We will show that for every $u \in \mathbb R / \mathbb Z$ we have
\begin{equation*}
 \lapv g' \in L^{p} ((u-1/20,u+1/20))
\end{equation*}
for a $p>2$. Due to the invariance of the problem under shifting the parametrization, it is enough to show this 
for $u=1/2$, i.e.
\begin{equation} \label{eq:NeedToProveRegOnlyLocally}
 \lapv g' \in L^{p} ((9/20,11/20)).
\end{equation}

Let $\laps{s} \equiv (- \partial_x^2)^{s/2} \equiv (-\lap)^{\frac{s}{2}}$ be the fractional Laplacian on $\R$. The inverse, $\lapms{s} \equiv (-\lap)^{-\frac{s}{2}}$, $s \in (0,1)$, is the Riesz potential by
\begin{equation}\label{eq:fraclap-}
 \lapms{s} f(x) = c_s \intl_{\R} \frac{f(y)}{\abs{x-y}^{1-s}}\ dy.
\end{equation}
In case of positive powers $s$ of the Laplacian $\laps{s}$, $s\in(0,1)$, we use the corresponding formula
\begin{equation}\label{eq:fraclap+}
 \laps{s} f(x) = \tilde c_s \intl_{\R} \frac{f(y)-f(x)}{\abs{x-y}^{1+s}}\ dy.
\end{equation}
For a detailed introduction to the fractional Laplacian we refer to, e.g., \cite{DiNezza}, \cite[Section~2.5]{Schikorra-Doktor}.

To switch from the circle to the real line, we interpret
functions on $\mathbb R / \mathbb Z$ as functions on $\mathbb R$ which are periodic 
with period $1$. We then choose a cutoff function $\eta \in C_0^\infty(\R)$, 
$\eta \equiv 1$ on $[-3,4]$ and consider
\begin{equation}\label{eq:g=g}
  g = \eta \cdot \gamma
\end{equation}
instead of $\gamma$.
We will show that $\lapv g' \in L^p((17/40,23/40),\R^\tdim)$,
for some $p > 2$ to conclude the proof.

Before we begin to outline the structure of the proof, let us shortly recapitulate the notion of 
Lorentz spaces and the main properties we are going to use in this article. For a measurable 
function $f: \Omega \rightarrow \mathbb R$ and $\Omega \subset \mathbb R$
one considers the \emph{decreasing rearrangement}
\begin{equation*}
 f^\ast (t):= \inf\left\{s>0: {\mathcal L}^1(\{|f|> s\})\leq t\right\}
\end{equation*}
where $\mathcal L^1$ denotes the Lebesgue measure. We define
\begin{equation*}
 |f|_{(p,q),\Omega} := \begin{cases}
                      \left(\int_0^\infty (t^{1/p} f^\ast (t)) ^q \frac {dt}{t}\right)^{1/q}  
			&\text{if } p,q \in[1, \infty),\\
                      \sup_{t>0} t^{1/p} f^\ast (t)&\text{if } q=\infty.
                     \end{cases}
\end{equation*}

To prevent technical problems, unless $p \in (1,\infty)$
we will only take the spaces  $L^{1,1} = L^1$ and $L^{\infty,\infty} =  L^\infty$
into consideration.

Furthermore, $\abs f_{(p,p)}\approx\norm f_{L^p}$ for all $p\in[1,\infty]$.
If $\Omega = \mathbb R$ we will omit $\Omega$ in the notation.
Though $|\cdot|_{(p,q), \Omega}$ is not a norm, as it does not obey the triangle inequality, there
is a norm $\|\cdot\|_{(p,q), \Omega}$ on the Lorentz spaces which is equivalent to $|\cdot|_{(p,q),\Omega}$.
These norms satisfy a \emph{H\"older inequality}, i.e., for $p_1,p_2,p \in [1, \infty)$
and $q_1,q_2,q \in [1,\infty]$ with $1/p_1 + 1/p_2 = 1/p$ and $1/q_1 + 1/q_2 = 1/q$ we
have 
\begin{equation}\label{eq:hoelder}
  \|fg\|_{(p,q), \Omega} \aleq \|f\|_{(p_1,q_1), \Omega}\ \|g\|_{(p_2,q_2), \Omega}.
\end{equation}
For $p_1, p_2, p \in (1,\infty)$ and $q_1,q_2 \in [1, \infty]$
with $1/p_1 + 1/p_2 = 1 /p + 1$ and $1/q_1 +1/q_2 = 1/q$ we have the \emph{Young-O'Neil inequality} \cite{Hunt66}
\begin{equation}\label{eq:young}
 \|f \ast g\|_{(p,q), \Omega} \aleq \|f\|_{(p_1,q_1),\Omega}\; \|g\|_{(p_2,q_2), \Omega}.
\end{equation}
Furthermore, we have the \emph{Sobolev inequality} 
\begin{equation*}
 \|\lapms{s} f\|_{(p^\ast,q)} \aleq \|f\|_{(p,q)}
\end{equation*}
for all $s\ge0$, $p\in [1,\infty), q \in [1, \infty]$ and $p^\ast :=  \frac {p}{1-sp} \in [1,\infty)$.
Further information and proofs can be found in~\cite{Hunt66,GrafakosC,Tartar2007}.

The main reason for using Lorentz spaces in the context of critical equations, i.e. equations to which
standard Gagliardo-Nirenberg-Sobolev embeddings cannot be applied to gain regularity, is the following fact.
Although for functions $f$ the $L^2$-norm of $\lapv  f$ does not control the 
$L^\infty$-norm of $f$, the $L^{2,1}$ norm does, i.e. we have the estimate
\begin{equation*}
  \|f\|_{\infty} \aleq \|\lapv   f\|_{(2,1)}.
\end{equation*}
We will also need this in the more general form
\begin{equation*}
 \|f\|_{\infty} \aleq \|\laps{s} f\|_{(\frac 1 s,1)} \quad \mbox{for all $s\in(0,1)$.}
\end{equation*}
In order to prove our regularity result, we will prove a Dirichlet growth theorem for the 
weak $H^{1/2}$-energy of $\gamma'$ on balls in a manner comparable to \cite{DR11Sphere,SchikorranharmSphere}
 -- which are as well in the setting of sphere-valued mappings.
In contrast to these papers, the techniques from \cite{DR11Man,DL11MannD,SchikorraIntBoundFrac11}
deal with a more general setting, but have to work (as we will here) with estimates
of the $L^{2,\infty}$-norm instead of the $L^2$-norm.
Note nevertheless that our right-hand side is very different from their's. In order to obtain the estimates
of the norms $\vrac{\lapv \dsolnt}_{(2,\infty)}$ on small balls, we will have to use new
arguments.

To prove the regularity theorem, we begin with an approach appearing in \cite{DR11Sphere, SchikorranharmSphere} and
divide $\lapv g'$ into the part parallel to $g'$ (and thus \emph{normal} to the sphere $\S^{n-1}$) and
the term perpendicular to $g'$ (\emph{tangential} to the sphere). More precisely, we use that for $p \in [1,\infty)$, $q \in [1,\infty]$
\begin{equation}\label{eq:differencegtanggorth}
 \vrac{\lapv \dsolnt}_{(p,q),B_r} \aleq \vrac{\sbrac{\dsolnt,\lapv \dsolnt}}_{(p,q),B_r} 
  + \sup_{\omega_{ij}} \vrac{\dsolnt_j \omega_{ij}\lapv \dsolnt_i}_{(p,q),B_r},
\end{equation}
where the supremum is over $\omega_{ij} = -\omega_{ji} \in\set{-1,0,1}$.
For a detailed version of this linear algebraic fact, the interested reader is referred to the appendix of \cite{DLSnalphaSphere}.

\subsection{Estimate of the normal part}

We have
\begin{equation}\label{eq:orthogpartHdec}
\sbrac{\dsolnt, \lapv \dsolnt} = -\fracm{2} H_{\frac{1}{2}}(\dsolnt,\dsolnt) + \fracm{2} \lapv \abs{\dsolnt}^{2},
\end{equation}
where
\begin{equation}\label{eq:H}
 H_{s}(a,b) := \laps{s}(ab) -  a\ \laps{s} b -  b\ \laps{s} a.
\end{equation}
Note that for any $s \in (0,1)$, we have
\begin{equation}\label{eq:Dg2}
\laps{s} \sabs{\dsolnt}^{2} \ \refeq{g=g} \ \laps{s} {\eta}^{2} + 2 \laps{s} ({\eta}' \g' g)
+ \laps{s}(|{\eta}'|^2 |\g|^2)
\in L^{\infty}((0,1)). 
\end{equation}
In fact, $\laps{s}\eta^2\in L^\infty$ by interpolation inequalities. For the remaining terms
we use the quasi-locality, Lemma~\ref{lem:QuasiLocality}, and the support of $\eta$ and $\eta'$.

As in \cite{SchikorraIntBoundFrac11} we will use pointwise estimates for $H_s$ and some quantitative version of
the quasi-locality to estimate the normal part of 
$\lapv   g'$:  

\begin{lemma}[Normal part]\label{la:tangentialpart}
For any $s \in [0,\fracm{2})$ there exists $\theta > 0$ such that for any $B_r \subset [0,1]$, $\Lambda > 4$
\begin{equation}\label{eq:orthogpartest}
 \vrac{\laps{s}\sbrac{ \dsolnt, \lapv \dsolnt }}_{(\frac{2}{1+2s},\infty),B_r} 
  \aleq \vrac{\lapv \dsolnt}_{(2,\infty),B_{\Lambda r}}^2 + \ \Lambda^{-\theta} \ \vrac{\lapv \dsolnt}_{(2,\infty),\R}\ 
  \ \sum_{k=1}^\infty 2^{-\theta k}\ \vrac{\lapv \dsolnt}_{(2,\infty),B_{2 ^k\Lambda r}}+r^{\frac12+s}.
\end{equation}
\end{lemma}

For the readers' convenience, a proof will be given in the appendix.

\subsection{Estimate of the tangential part}

It then remains to estimate the part normal to $g'$ (tangential to the sphere), i.e.  for $\omega_{ij} = - \omega_{ji} \in \set{-1,0,1}$, 
$1 \leq i,j \leq n$ we need to estimate 
suitable norms on small balls of the term $\dsolnt_j \omega_{ij} \lapv \dsolnt_i$. 

We have
\begin{align}
\intl_{\R} \dsolnt_j \omega_{ij} \lapv \dsolnt_i\ \lapv \varphi &=
 \intl_{\R} \lapv \dsolnt_i\ \lapv \brac{\omega_{ij} \dsolnt_j \varphi} 
  -\intl_{\R} \lapv \dsolnt_i\ \omega_{ij} \lapv \dsolnt_j\ \varphi 
  -\intl_{\R} \lapv \dsolnt_i\ \omega_{ij} H_{1/2}(\dsolnt_j,\varphi) \nonumber\\
 & = \intl_{\R} \lapv \dsolnt_i\ \lapv \brac{\omega_{ij} \dsolnt_j \varphi} 
  -\intl_{\R} \lapv \dsolnt_i\ \omega_{ij} H_{1/2}(\dsolnt_j,\varphi)\label{eq:antisympde}
\end{align}
where we have used that due to $\omega_{ij} = - \omega_{ji}$ the second term on the right-hand side of the first line
vanishes. 

The second term can be estimated analogously to similar terms in \cite{DR11Man,DL11MannD,SchikorraIntBoundFrac11} using again quasi-locality together with Sobolev embeddings.

\begin{lemma}[Tangential part] \label{lem:CriticalTerm}
There is $\theta >0$, $s_0\in(0,\frac12)$ such that for all 
$\phi \in C_0^\infty(B_r)$, $\Lambda > 16$, $s \in [0,s_0)$  
we have
\begin{multline*}
 \int \lapv   g_i'\omega_{ij} H_{1/2}(g_j',\phi) dx \\\aleq 
  \|\Dsymbol^{-s+1/2}\phi\|_{(\frac 2 {1-2s},1)}
  \left(\|\lapv  g'\|_{(2,\infty),B_{\Lambda r}}^2 + \Lambda^{-\theta} \|\lapv  g'\|_{(2,\infty),\R}
    \sum_{k=1}^\infty 2^{-\theta (k-1)} \|\lapv   g'\|_{(2,\infty), B_{2^k\Lambda r}} \right).
\end{multline*}
\end{lemma}

A proof is provided in the appendix.

It remains to estimate the first term on the right-hand side of Equation~\eqref{eq:antisympde}
for which we will use the Euler-Langrange equation~\eqref{eq:knotPDE}.

Combining this equation with the formula
\begin{multline} \label{eq:FormulaForNorm}
 \int_{\mathbb R } \lapv f_1'\ \lapv f_2' \\
 = c \lim_{\eps\searrow0}\int_{\mathbb R} \int_{\mathbb R\setminus(-\eps,\eps)} 
  \left( \left \langle f_1'(u), f'_2 (u) \right\rangle- 
  \frac{\left\langle f_1(u+w)-f_1(u), f_2(u+w)-f_2(u)\right\rangle }{w^2} \right) 
 \frac {dw}{w^2} du
\end{multline}
due to He \cite[Proposition 2]{He1999} we get the following estimate which 
contains all the information of the Euler-Langrange equation we need to proceed in the proof:

%
\begin{lemma}[Essential estimate of the Euler-Lagrange equation]\label{la:lhs}
  There is a constant $C < \infty$ such that 
  \begin{equation}\label{eq:el-est}
   \int_{\mathbb R} \left\langle \lapv g', \lapv   \phi \right \rangle 
   \leq C \int_{\mathbb R} |\phi(u)| \Gamma(u) du  + C \|\phi\|_{L^2}
  \end{equation}
 for any $\phi \in C^{\infty}_0 ((4/10,6/10), \mathbb R^n)$ with 
 $\left\langle\phi,\gamma'\right\rangle \equiv 0$ where
 \begin{equation*}
  \Gamma(u):= \int_{(-1,1)^3} \int_{-1/4}^{1/4} \frac{|g'(u)-g'(u+s_2 w)| \ 
  |g'(u+s_3 w)-g'(u+s_4 w)|^2} {|w|^2}\ \ dw\ ds.
 \end{equation*}
\end{lemma}

The heart of the proof of Theorem~\ref{thm:FirstStep} is the 
following pointwise estimate of the most problematic term $\Gamma(u)$, which permits to localize it and which afterwards will be transformed 
into a bound of its $L^1$-norm.

\begin{lemma}[Estimate of the critical term]\label{la:T2est}
We have
\begin{equation}\label{eq:Gammauest}
 \Gamma(u) \aleq \abs{\Dsymbol^{-\frac{1}{12}} {\vert \lapv \dsolnt\vert}(u)}^2\ \Dsymbol^{-\frac{1}{3}}{\vert \lapv \dsolnt\vert}(u)+\Dsymbol^{-\frac{1}{12}} {\vert \lapv \dsolnt\vert}(u)\ \Dsymbol^{-\frac{1}{4}}\abs{\Dsymbol^{-\frac{1}{12}} {\vert \lapv \dsolnt\vert}(u)}^2
\end{equation}
almost everywhere.
\end{lemma}

In fact, Lemmata~\ref{la:lhs} and~\ref{la:T2est} provide the essential new estimates in this article
on which our entire reasoning relies crucially.
Using these and the above-mentioned improved Sobolev embeddings for
Lorentz spaces, we are ready to show the following estimate. This will allow us to 
prove a Dirichlet growth of the $L^{2, \infty}$-norm of $\lapv   g'$.

\begin{lemma}\label{la:rhsests}
There exist $R > 0$, $s_0 \in (0,\fracm{2})$, and $\sigma>0$
such that for any $\Lambda > 2$, $s \in [0,s_0)$, 
$B_{\Lambda r} \subset (\frac{5}{10},\frac{6}{10})$, 
$r \in (0,\Lambda^{-1}R)$ and $\varphi \in C_0^\infty(B_r)$ we have
\begin{align*}
 &\int \dsolnt_i \omega_{ij} \lapv \dsolnt_j\ \lapv \varphi \\
 &\aleq \vrac{\lapms{s} \lapv \varphi}_{(\frac{2}{1-2s},1)} \brac{r^\sigma + \vrac{\lapv \dsolnt}_{(2,\infty),B_{\Lambda r}}^2+ \Lambda^{-\theta} \vrac{\lapv \dsoln}_{(2,\infty),\R}\sum_{k=1}^\infty 2^{-\theta k}\ \vrac{\lapv \dsoln}_{(2,\infty),B_{2^k\Lambda r}}}.  
\end{align*}
\end{lemma}

\subsection{Conclusion of the proof of Theorem~\ref{thm:FirstStep}}

Combining Lemma~\ref{la:rhsests} with Proposition~\ref{pr:localcontrolelliptic}, we get 
\begin{align}
 &\vrac{\laps{s}\brac{\dsolnt_i \omega_{ij} \lapv \dsolnt_j}}_{(\frac{2}{2s+1},\infty),B_{r}} \nonumber\\
&\aleq 
\brac{\br{\Lambda^{2}r}^\sigma + \vrac{\lapv \dsolnt}_{(2,\infty),B_{\Lambda^{3}r}}^2+ \vrac{\lapv \dsoln}_{(2,\infty),\R} \Lambda^{-\theta} \sum_{k=1}^\infty 2^{-\theta k}\ 
\vrac{\lapv \dsoln}_{(2,\infty),B_{\Lambda^{3} 2^k r}}} \nonumber
\\ &\quad {}+ \Lambda ^{-\theta}r^{-s} \sum_{k=1}^\infty 2^{-\theta k}\ \vrac{\lapv \dsoln}_{(\frac2{1+2s},\infty),B_{\Lambda 2^k r}}\nonumber\\
&\aleq\br{\rule{0ex}{4ex}\cdots} +
\Lambda^{s-\theta} \sum_{k=1}^\infty 2^{(s-\theta) k}\ \vrac{\lapv \dsoln}_{(2,\infty),B_{\Lambda 2^k r}}
\label{eq:nonorthogpartest1}
\end{align}
for some $\theta >0$ uniformly in $\Lambda$.

Let us first use this for $s = 0$, to get in view of \eqref{eq:differencegtanggorth} 
and Lemma \ref{la:tangentialpart} that for all $\varepsilon>0$ we have for sufficiently small
 $r>0$, $B_r \subset \brac{\frac{4}{10},\frac{6}{10}}$  and big enough $\Lambda$
\[
 \vrac{\lapv \dsolnt}_{(2,\infty),B_r} \leq C\br{\br{\Lambda^{2}r}^\sigma + 
  \varepsilon\ \vrac{\lapv \dsolnt}_{(2,\infty),B_{\Lambda^{3} r}}+ (1+\vrac{\lapv \dsoln}_{(2,\infty),\R}) \Lambda^{-\theta} \sum_{k=2}^\infty 2^{-\theta k}\ \vrac{\lapv \dsoln}_{(2,\infty),B_{2^k \Lambda^{3} r}}}
\]
uniformly in $\Lambda$ and $\varepsilon$.

Let us fix such an $r_0$, and consider the above equation for $\Lambda = 2^{m/3}$,
$r = 2^{-m-k} r_0$ (w.l.o.g. $r_0 = 1$).  Setting $b_0 := \|\lapv   g'\|_{(2,\infty)}$ 
and $b_k:= \|\lapv  g'\|_{(2,\infty),B_{2^{-k}}}$ the above estimate gives
\begin{equation*}
 b_{k+m} \leq C 2^{-\sigma (k+m/3)} + \varepsilon b_{k} + 
 C 2^{-\theta m/3}\br{\sum_{l=1}^{k} 2^{-\theta l}b_{k-l}
 +C2^{-\theta k}}
\end{equation*}
for every $k \in \mathbb N_0$ where $C< \infty$ does not depend on $k$.

Using the iteration argument, Lemma~\ref{lem:IterationLemma}, leads to
\begin{equation} \label{eq:morreynormlapvdsolnt}
 \vrac{\lapv \dsolnt}_{(2,\infty),B_r} \leq C_{\soln}\ r^{\tilde{\sigma}}
\end{equation}
for all $B_{r} \subset \brac{\frac{9}{20},\frac{11}{20}}$
 and $r$ small enough.


This immediately implies H\"older regularity of the solution $g'$. Instead, let us iterate
the above argument to get $\lapv g' \in L^p_{loc}((9/20,11/20))$, for some $p > 2$.

Let us assume without loss of generality that $\tilde{\sigma} < \frac{\theta}{2}$. 
Then, plugging \eqref{eq:morreynormlapvdsolnt} into \eqref{eq:nonorthogpartest1} one obtains
\[
 \vrac{\laps{s}\brac{\dsolnt_i \omega_{ij} \lapv \dsolnt_j}}_{(\frac{2}{2s+1},\infty),B_{r}} 
 \aleq C_{\Lambda,\soln,s,R}\ r^{\min\{\sigma,\tilde{\sigma}, \tilde{\sigma}-s,\theta/2\}}
\]
for small enough $0<s<(0,\tilde\sigma)$, $r>0$, and $B_r \subset [9/20, 11/20]$.
Here we have used that
$2^{k}\Lambda^{3}r\ge R_{0}$ for some $R_{0}>0$ leads
to $2^{-\theta k}\le 2^{-\theta k/2}\br{\frac{\Lambda^{3}}{R_{0}}}^{\theta/2}r^{\theta/2}$.
So the series in~\eqref{eq:nonorthogpartest1}
converge and are bounded by some small positive power of $r$.


On the other hand, \eqref{eq:morreynormlapvdsolnt} and \eqref{eq:orthogpartest} 
together imply for any small enough $r >0$, $\Lambda := r^{-\fracm{2}}$,
and $B_{r}\subset\brac{\frac{9}{20},\frac{11}{20}}$  
\[
\begin{ma} 
 \vrac{\laps{\fracm{4}}\sbrac{ \dsolnt, \lapv \dsolnt }}_{(\frac{4}{3},\infty),B_r}  
&\overset{\eqref{eq:orthogpartest}}{\aleq}& \vrac{\lapv \dsolnt}_{(2,\infty),B_{r^{\fracm{2}}}}^2 + r^{\frac{\theta}{2}}\ \vrac{\lapv \dsolnt}_{(2,\infty),\R}^2 + r^{\frac34}\\
&\overset{\eqref{eq:morreynormlapvdsolnt}}{\leq}& C_{\soln}\ \brac{r^{\tilde{\sigma}} + r^{\frac{\theta}{2}} + r^{\frac34}}.
\end{ma}
\]
That is, $\laps{s}\brac{\dsolnt_i \omega_{ij} \lapv \dsolnt_j}$, and $\laps{\fracm{4}}\sbrac{ \dsolnt, \lapv \dsolnt }$ both belong locally to a Morrey space
$\mathcal L_{loc}^{(p,q),\lambda}$ on $\brac{\frac{9}{20},\frac{11}{20}}$.  
More precisely, for some $\lambda \in (0,1)$ we have
\[
 \begin{ma}
\laps{s}\brac{\dsolnt_i \omega_{ij} \lapv \dsolnt_j}    &\in&  \mathcal{L}^{(\frac{2}{2s+1},\infty),\lambda}_{loc}\brac{\frac{9}{20},\frac{11}{20}}, \\
\laps{\fracm{4}}\sbrac{ \dsolnt, \lapv \dsolnt } &\in& \mathcal{L}^{(\frac{4}{3},\infty),\lambda}_{loc}\brac{\frac{9}{20},\frac{11}{20}}.
 \end{ma}
\]

The boundedness of Riesz potentials on Morrey spaces, as shown in \cite{Adams75}, 
implies that for some $p > 2$ 
\[
 \dsolnt_i \omega_{ij} \lapv \dsolnt_j,\ \sbrac{ \dsolnt, \lapv \dsolnt } \in L^{p}_{loc}\brac{\frac{9}{20},\frac{11}{20}}.
\]

Together, using \eqref{eq:differencegtanggorth}, we have shown
\[
 \lapv \dsolnt_j \in L^p_{loc}\brac{\frac{9}{20},\frac{11}{20}},
\]
which finishes the proof of Theorem~\ref{th:reg:1alpha}.\qed

\subsection{
Proof of Lemma~\ref{la:lhs}}

The main idea is to use $\phi$ as the derivative of a test function for the 
Euler-Lagrange equation. Of course this is not possible generally, but with some
precaution we can actually do it up to a benign error term:
For $\phi \in C^\infty_0 ((4/10,6/10))$ we set
\begin{equation*}
   h(u):= \int_0^u \phi(v) dv - a u 
\end{equation*}
where $a=\int_0^1 \phi(v) dv$  is chosen such that $ h(0)= h(1)$. 
Hence, if we set
\begin{align*}
h_\pi(k+t) := h(t), \quad \forall t\in [0,1) , k \in \mathbb Z, \\
\phi_\pi(k+t) := \phi(t), \quad \forall t\in [0,1) , k \in \mathbb Z,
\end{align*}
$h_\pi$ is a smooth one periodic function satisfying
\begin{equation*}
    h_\pi'(u) =  \phi_\pi  - a.
\end{equation*}
and, as we assume in Lemma~\ref{la:lhs} that $\left\langle\phi,\gamma'\right\rangle \equiv 0$, also
\begin{equation*}
 \left\langle \phi_\pi , \gamma'\right \rangle \equiv 0.
\end{equation*}

Since $\gamma$ is a stationary point of the M\"obius energy, testing the 
equation \eqref{eq:knotPDE} with $h_\pi$, (recall \eqref{eq:defT1}, \eqref{eq:defT2},\eqref{eq:defQeps}) 
\begin{equation}\label{eq:ourELeqn}
 Q(\gamma,h_\pi) := \lim_{\varepsilon \to 0} Q_\varepsilon(\gamma,h_\pi)
= T_1(\gamma,h_\pi) + T_2(\gamma,h_\pi).
\end{equation}

As $\phi_\pi$ is perpendicular to $\gamma'$ we can estimate the term $T_1(\gamma,h_\pi)$
by
\begin{align}
 T_1(\gamma,h_\pi) &= \int_{-1}^1 \int_{-1/2}^{1/2}
   \left\langle \gamma'(u),\phi_\pi(u) - a \right\rangle 
  \left(\frac 1 {|\gamma(u+w)- \gamma(u)|^2} - \frac 1 {w^2} \right)\nonumber\\
  &\leq |a|\ E^{(2)}(\gamma) \leq C \|\phi\|_{L^1} \leq C \|\phi\|_{L^2}.\label{eq:toRT1}
\end{align}
As for the remaining terms $Q(\gamma,h_\pi)$ and $T_2(\gamma,h_\pi)$, we will identify them essentially with the left-hand side of \eqref{eq:el-est} and the $\Gamma$-term on the right-hand side of \eqref{eq:el-est}, respectively. 
A technical detail one has to take into account here, is that the domain of \eqref{eq:ourELeqn} is the torus $\R / \Z$, whereas the respective domain in \eqref{eq:el-est} is the real line $\R$.
To estimate the other terms, let us introduce for 
$f_1, f_2 \in H^{3/2}_{loc}(\mathbb R, \mathbb R^n)$
the operators
\begin{align*}
 \tilde{Q}_\varepsilon (f_1,f_2)& :=  
 \int_{0}^{1}\int_{\tilde I_\varepsilon}
  \left(\left\langle f_1'(u),\ f_2'(u) \right\rangle -
  \frac { \left\langle f_1(u+w)-f(u), f_2(u+w) -f_2(u) 
  \right\rangle}{w^2} \right) \frac {dw}{w^2}\ du\\
& = \int_{0}^{1}\int_{\tilde I_\varepsilon} \int_0^1 \int_0^1
  \frac {\left( \left\langle f_1'(u),\ f_2'(u) \right\rangle- \left\langle f'_1(u+s_1w), f'_2(u+s_2w)
  \right\rangle	\right) }{w^2} ds_1 ds_2 dw\ du,
\end{align*}
where $\tilde I_\varepsilon := [-1/4, 1/4]\setminus [-\varepsilon, \varepsilon]$,
\[ \tilde Q := \lim_{\eps\searrow0} \tilde Q_\eps, \]
and
\begin{equation}
\begin{aligned}
 &\tilde {T}_2(f_1,f_2) \\ &:=  \int_{0}^{1}\int_{-1/4}^{1/4}
  \left\langle f_1(u+w)-f_1(u),\ f_2(u+w)-f_2(u) \right\rangle \left(\frac 1  {|\gamma(u+w)-\gamma(u)|^4}-
 \frac1{|w|^4}\ \right) dw\ du \\
&= \int_{0}^{1}\int_{-1/4}^{1/4} \int_0^1 \int_0^1 
  \left\langle f_1'(u+s_1 w),\ f_2'(u+s_2 w) \right\rangle w^2\left( \frac 1 {|\gamma(u+w)-\gamma(u)|^4}
- \frac1{|w|^4}\ \right) ds_1\ ds_2\ dw\ du.
\end{aligned}
\end{equation}

Recall that $[-\frac{1}{2},\frac{1}{2}]\backslash (-\varepsilon,\varepsilon)$ is used for the definition of $Q$ 
and $\tilde I_\varepsilon := [-1/4, 1/4]\setminus [-\varepsilon, \varepsilon]$ is used for $\tilde{Q}$, 
so the difference only contains the set where $\abs{w} > \fracm{4}$, thus $\abs{w}^{-2}$ is not singular. Quantitatively, this reads as
\begin{equation} \label{eq:toRtilde}
 |\tilde Q(\gamma,h) - Q(\gamma,h)| + |\tilde T_2 (\gamma,h) - T_2(\gamma,h)| 
\leq C \|\gamma'\|_{L^2} \|h'\|_{L^2} \leq C \|\gamma'\|_{L^2} \|\phi\|_{L^2},
\end{equation}
where we have used that $\gamma$ is bi-Lipschitz in order to deal with $T_2$.

We now compute
\begin{align}
 \tilde{Q}_\eps(\gamma,h_\pi) &=\int_{0}^{1}\int_{\tilde I_\varepsilon} \int_0^1 \int_0^1
  \br{\left\langle \gamma'(u),\ h_\pi'(u)\right\rangle - \left\langle \gamma'(u+s_1w), h'_\pi(u+s_2w)
  \right\rangle} \d s_1\d s_2\frac{dw}{w^2}\ du \nonumber  \\
 &= \int_{0}^{1}\int_{\tilde I_\varepsilon} \int_0^1 \int_0^1
  \br{\left\langle \gamma'(u),\ \phi_{\pi} \right\rangle - \left\langle \gamma'(u+s_1w),  \phi_\pi(u+s_2w)
  \right\rangle} \d s_1\d s_2\frac{dw}{w^2}\ du \nonumber \\
&\quad - \int_{0}^{1}\int_{\tilde I_\varepsilon} \int_0^1 \int_0^1
  \br{\left\langle \gamma'(u),\ a \right\rangle -\left\langle \gamma'(u+s_1w), a
  \right\rangle} \d s_1\d s_2\frac{dw}{w^2}\ du \label{eq:qtildenullterm} \\
&= \int_{0}^{1}\int_{\tilde I_\varepsilon} \int_0^1 \int_0^1
  \br{\left\langle g'(u),\ \phi(u) \right\rangle - \left\langle g'(u+s_1w), \phi(u+s_2w)
  \right\rangle} \d s_1\d s_2\frac{dw}{w^2}\ du . \label{eq:qtilderhs}
\end{align}
Here we have used that $\int_0^1 \gamma'(u+\tilde w) du =0$ for all $\tilde{w}$, i.e., the term \eqref{eq:qtildenullterm} is constantly zero for any $w$. The term \eqref{eq:qtilderhs} essentially is the $L^2$-pairing of $\lapv g'$ and $\lapv \phi$: More precisely we will show
\begin{equation} \label{eq:toRQ}
 |\tilde{Q}(\gamma,h_\pi) - \int_{\mathbb R} \lapv g' \lapv \phi | \leq C_\gamma \|\phi\|_{L^2}.
\end{equation}
\begin{proof}[Proof of \eqref{eq:toRQ}]
%
Using~\eqref{eq:FormulaForNorm}, we have for any $f_1,f_2 \in C_0^\infty(\R)$ as $\eps \to 0$
\begin{equation}\label{eq:Heidentity}
\begin{aligned}
  &\int_{\mathbb R} \lapv f_1' \lapv f_2' du + o(1) \\
 &= c\int_{\mathbb R} \int_{\mathbb R \setminus [-\varepsilon, \varepsilon]} \int_0^1 \int_0^1
  \left(\left\langle f_1'(u), f_2'(u) \right\rangle
 - \left\langle f'_1(u+s_1w), f'_2(u+s_2w) \right\rangle  \right) 
\d s_1\d s_2\frac {dw}{|w|^2} du.
\end{aligned}
\end{equation}
We now connect and \eqref{eq:Heidentity} and \eqref{eq:qtilderhs}. The technical problem is, that the integral of $\phi$ is not a feasible test-function for \eqref{eq:Heidentity}. Let therefore for $\eta_{[-10,10]} \in C_0^\infty ((-11,11))$ and $\eta_{[-10,10]} \equiv 1$ in $[-10,10]$,
\[
 \psi := \brac{u \mapsto \eta_{[-10,10]} \int_0^u \phi} \in C_0^\infty(\R).
\]
Thus, $\psi$ is a feasible test-function for \eqref{eq:Heidentity}, which $\phi$ is not. Moreover,
\begin{equation}\label{eq:l42:psip}
 \psi' = \eta_{[-10,10]} \phi + \eta'_{[-10,10]} \int_0^u \phi \overset{\supp \phi}{\equiv } \phi + \eta'_{[-10,10]} \int_0^u \phi.
\end{equation}
We thus arrive at
\begin{align*}
 \tilde{Q}_\varepsilon(\gamma,h) \;
 &\refeq{qtilderhs}
 \int_{0}^{1}\int_{\tilde I_\varepsilon} \int_0^1 \int_0^1
  \br{\left\langle g'(u),\ \phi(u) \right\rangle - \left\langle g'(u+s_1w), \phi(u+s_2w)
  \right\rangle} \d s_1\d s_2\frac{dw}{w^2}\ du\\
 &\stackrl{$\supp \phi$, $\abs{w}$,\eqref{eq:l42:psip}}\qquad
 \int_\R \int_{\tilde I_\varepsilon} \int_0^1 \int_0^1
  \br{\left\langle g'(u),\ \psi'(u) \right\rangle - \left\langle g'(u+s_1w), \psi'(u+s_2w)
  \right\rangle} \d s_1\d s_2\frac{dw}{w^2}\ du\\ 
 &= \int_\R \int_{\R \backslash [-\varepsilon,\varepsilon]} \int_0^1 \int_0^1
  \br{\left\langle g'(u),\ \psi'(u) \right\rangle - \left\langle g'(u+s_1w), \psi'(u+s_2w)
  \right\rangle} \d s_1\d s_2\frac{dw}{w^2}\ du\\
&\qquad{} - \int_\R \int_{\abs{w} > \fracm{4}} \int_0^1 \int_0^1
  \br{\left\langle g'(u),\ \psi'(u) \right\rangle - \left\langle g'(u+s_1w), \psi'(u+s_2w)
  \right\rangle} \d s_1\d s_2\frac{dw}{w^2}\ du\\
  &\refeq{Heidentity}
  \int \langle \lapv g', \lapv \psi' \rangle\\
&\qquad{} - \int_\R \int_{\abs{w} > \fracm{4}} \int_0^1 \int_0^1
  \br{\left\langle g'(u),\ \psi'(u) \right\rangle - \left\langle g'(u+s_1w), \psi'(u+s_2w)
  \right\rangle} \d s_1\d s_2\frac{dw}{w^2}\ du\\
  &\qquad{} + o(1)\\
&\refeq{l42:psip}
\int \langle \lapv g', \lapv \phi \rangle + \int \langle \lapv g', \lapv (\eta'_{[-10,10]} \int_0^{\cdot} \phi) \rangle\\
&\qquad{} - \int_\R \int_{\abs{w} > \fracm{4}} \int_0^1 \int_0^1
  \br{\left\langle g'(u),\ \psi'(u) \right\rangle - \left\langle g'(u+s_1w), \psi'(u+s_2w)
  \right\rangle} \d s_1\d s_2\frac{dw}{w^2}\ du\\
  &\qquad{} + o(1).
\end{align*}
Now, by usual interpolation and/or imbedding of Sobolev spaces, see, e.g. \cite{Tartar2007,Schikorra-Doktor},
\[
 \vrac{ \lapv f}_{2,\R} \aleq \vrac{f}_{2,\R} + \vrac{f'}_{2,\R},
\]
we have
\[
 \| \lapv (\eta'_{[-10,10]} \int_0^{\cdot} \phi) \|_{2,\R} \aleq \vrac{\phi}_1 + \vrac{\phi}_2 \overset{\supp \phi}{\aleq} \vrac{\phi}_2.
\]
Moreover,
\[
\begin{ma}
 &&\int_\R \int_{\abs{w} > \fracm{4}} \int_0^1 \int_0^1
  \br{\left\langle g'(u),\ \psi'(u) \right\rangle - \left\langle g'(u+s_1w), \psi'(u+s_2w)
  \right\rangle} \d s_1\d s_2\frac{dw}{w^2}\ du \\
  &\aleq& \int_{\abs{w} > \fracm{4}} \abs{w}^{-2}\ \vrac{g'}_{2} \vrac{\psi'}_{2}\\
  &\overset{\eqref{eq:l42:psip}}{\aleq}& \vrac{g'}_{2,\R} \brac{\vrac{\phi}_{1,\R} + \vrac{\phi}_{2,\R}} \overset{\supp g}{\aleq} C_\gamma\ \vrac{\phi}_{2,\R}.
 \end{ma}
\]
Thus, we have shown that \eqref{eq:toRQ} holds.
\end{proof}

To estimate $\tilde T_2 (\gamma,h)$ we calculate
\begin{align*}
 \tilde T_2(\gamma,h)
 &= \int_{0}^{1}\int_{-1/4}^{1/4} \int_0^1 \int_0^1
  \left\langle \gamma'(u+s_1 w),\  \phi_\pi(u+s_2 w) 
  + a\right\rangle w^2\cdot\\
  &\qquad\qquad\qquad\qquad{}\cdot\left( \frac 1 {|\gamma(u+w)-\gamma(u)|^4}-
 \frac1{|w|^4}\ \right) ds_1\ ds_2\ dw\ du\\
 &= \int_{0}^{1}\int_{-1/4}^{1/4} \int_0^1 \int_0^1
  \left\langle \gamma'(u+s_1 w),\ \phi(u+s_2 w)- \phi(u+s_1 w) 
+ a \right\rangle w^2 \cdot\\
&\qquad\qquad\qquad\qquad{}\left( \frac 1 {|\gamma(u+w)-\gamma(u)|^4}-
 \frac1{|w|^4}\ \right) ds_1\ ds_2\ dw\ du
\end{align*}
and using
\begin{align*}
 w^2 \left( \frac 1 {|\gamma(u+w)-\gamma(u)|^4}- \frac 1 {|w|^4} \right) 
  & = \frac {w^4} {|\gamma(u+w)-\gamma(u)|^4} 
      \left( \frac {1 - \frac {|\gamma(u+w)- \gamma(u)|^4} {w^4}}{w^2}\right) \\
 & \leq C \frac {1 - \frac {|\gamma(u+w)- \gamma(u)|^2} {w^2}}{w^2} \\
 &  = \frac C 2 \int_{0}^1 \int_0^1 \frac {|\gamma'(u+s_3 w) - \gamma'(u+s_4 w)|^2}{w^2}ds_3 \ ds_4
\end{align*}
we get
\begin{equation} \label{eq:toRT2}
\begin{aligned}
 &|\tilde T_2(\gamma,h)|  \leq \abs a E^{(2)} (\gamma) \\ &\quad  + \int_{0}^{1}\int_{-1/4}^{1/4} \int_{(0,1)^4}
 \frac{ |\gamma'(u+s_1 w)-\gamma'(u+s_2w)| \ |\gamma'(u+s_3 w) - \gamma'(u+s_4 w)|^2 \ 
|\phi(u+s_2 w)|}{w^2}  dsdw\ du \\
& \leq C \|\phi\|_{L^2} + \int_{0}^{1}\int_{-1/4}^{1/4} \int_{(-1,1)^3}
 \frac{ |\gamma'(u+s_1 w)-\gamma'(u)| \ |\gamma'(u+s_3 w) - \gamma'(u+s_4 w)|^2 \ 
|\phi(u)|}{w^2}  dsdw\ du.
\end{aligned}
\end{equation}

From \eqref{eq:toRtilde},\eqref{eq:toRT1},\eqref{eq:toRQ}, and \eqref{eq:toRT2} 
one gets the claim, since $\gamma' = g'$ on $[-1/4,5/4]$.


\subsection{
Proof of Lemma~\ref{la:T2est}}
Let
\[
 F(u) := \sabs{\lapv {\solnt'}(u)}.
\]
Since $g',\lapv g'\in L^2$, we obtain
\begin{align*}
 &{\solnt}'(x)-{\solnt}'(y) = \lapmv (\lapv {\solnt}')(x)-\lapmv (\lapv {\solnt}'(y)) \\
 &= c_{\frac12}\br{\int \sabs{x-\xi}^{-1+\fracm{2}}\ \lapv {\solnt}'(\xi)\ d\xi-\int \sabs{y-\xi}^{-1+\fracm{2}}\ \lapv {\solnt}'(\xi)\ d\xi},
\end{align*}
then
\[
 \sabs{{\solnt}'(x)-{\solnt}'(y)} \aleq \intl_{\R} \sabs{\sabs{\xi-x}^{-1+\frac{1}{2}}-\sabs{\xi-y}^{-1+\frac{1}{2}}}\ {F(\xi)}\ d\xi,
\]
and hence
\[
 \Gamma(u) \aleq \intl_{(-1,1)^3}\ \intl_{\R^3} \intl_{-1/4}^{1/4} {F(\xi_1)}\ {F(\xi_2)}\ {F(\xi_3)}\ k(\xi,u,s,w)
\ dw\ d\xi\ ds, 
\]
where for almost every $s = (s_1,s_2,s_3) \in (-1,1)^3$, $\xi = (\xi_1,\xi_2,\xi_3) \in \R^3$, $w \in \R$, $u \in \R$,
\begin{align}\label{eq:k}
 k(\xi,u,s,w) 
&= \frac{
m\brac{\xi_1-u,0,s_1,w}\ 
m\brac{\xi_2-u,s_2,s_3,w}\ 
m\brac{\xi_3-u,s_2,s_3,w}\ 
} {\sabs{w}^{2}}
\end{align}
and
\[
m(a,s,t,w) := \sabs{\sabs{a+sw}^{-\frac{1}{2}}-\sabs{a+tw}^{-\frac{1}{2}}}.
\]
The characteristic behavior of $k$ is as follows:
The factors $m(\cdot,\cdot,\cdot,w)$ will behave like $\abs w^\delta$ in a neighborhood of $w=0$
such that they somewhat absorb the singular behaviour of $\abs{w}^{-2}$, that is, $k$ becomes integrable around $w=0$. This is an effect very similar to the behaviour of $H_{s}(\cdot,\cdot)$, see Lemma~\ref{lem:EstimateForH12s}, as developed in \cite{SchikorraIntBoundFrac11}.

More precisely, we will derive the estimate
\begin{subequations}\label{eq:k-est}
\begin{align}
 &\int_{[-1,1]^3} {k(\xi,u,s,w)}\d s \nonumber\\
 &\aleq \abs w^{3\delta-2}\sabs{\xi_1-u}^{-1/2-\delta}
        \sum_{\br{\s(2),s(3)}\in\set{(2,3),(3,2)}}\sabs{\xi_{\sigma(2)}-u}^{-1/2-\delta}\int_{-1}^1\sabs{\xi_{\sigma(3)}-u+tw}^{-1/2-\delta}\d t
        \label{eq:typI}\\
 &\quad{}+\abs w^{3\delta-2}\sabs{\xi_1-u}^{-1/2-\delta}
        \int_{-1}^1\sabs{\xi_{2}-u+tw}^{-1/2-\delta}\sabs{\xi_{3}-u+tw}^{-1/2-\delta}\d t.
        \label{eq:typII}
\end{align}
\end{subequations}
We start with some abstract treatment of $m$.
\setlength{\multlinegap}{0ex}%
\begin{multline}\label{eq:est:chi1}
 \text{In \textbf{case~1},}\hfill
 \max\br{\sabs{a+sw},\sabs{a+tw}}\ge2\sabs{s-t}\abs w,\hfill
\end{multline}
we obtain
\begin{equation}\label{eq:est:chi1apsweqaptw}
 \sabs{a+sw} \aeq \sabs{a+tw}.
\end{equation}
Applying the mean value theorem, for any $\delta\in(0,1)$ we arrive at
\begin{align}
  m(a,s,t,w)
  &\aleq \max\br{\sabs{a+sw}^{-3/2},\sabs{a+tw}^{-3/2}}\sabs{s-t}\sabs{w} \nonumber\\
  &\aleq \max\br{\sabs{a+sw}^{-1/2-\delta},\sabs{a+tw}^{-1/2-\delta}}\max\br{\sabs{a+sw}^{-1+\delta},\sabs{a+tw}^{-1+\delta}}\sabs{s-t}\sabs{w} \nonumber\\
  &\refeq[\aleq]{est:chi1}\max\br{\sabs{a+sw}^{-1/2-\delta},\sabs{a+tw}^{-1/2-\delta}}\underbrace{\sabs{s-t}^\delta}_{\le2^\delta}\sabs{w}^\delta \nonumber\\
  &\refeq[\aleq]{est:chi1apsweqaptw}\min\br{\sabs{a+sw}^{-1/2-\delta},\sabs{a+tw}^{-1/2-\delta}}\sabs{w}^\delta. \label{eq:est:chi1roughest}
\end{align}
\begin{multline}\label{eq:est:chi2}
 \text{In \textbf{case~2},}\hfill
 \max\br{\sabs{a+sw},\sabs{a+tw}}\le2\sabs{s-t}\abs w,\hfill
\end{multline}
we immediately obtain
\begin{align}
  m(a,s,t,w)
  &\aleq \max\br{\sabs{a+sw}^{-1/2},\sabs{a+tw}^{-1/2}} \nonumber\\
  &\aleq \max\br{\sabs{a+sw}^{-1/2-\delta},\sabs{a+tw}^{-1/2-\delta}}\underbrace{\sabs{s-t}^\delta}_{\le2^\delta}\sabs{w}^\delta. \label{eq:est:chi2roughest}
\end{align}
We begin with the first factor in~\eqref{eq:k}. In case~1 we always have
\begin{equation}\label{eq:x11}
 m(\xi_1-u,0,s_1,w) \aleq \sabs{\xi_1-u}^{-1/2-\delta}\abs w^\delta.
\end{equation}
In case~2 we have either $\sabs{\xi_1-u+s_1w}\ge\tfrac12\sabs{\xi_1-u}$ which immediately results in~\eqref{eq:x11}
or the opposite
\begin{equation}\label{eq:x11a}
 \sabs{\xi_1-u+s_1w}\le\tfrac12\sabs{\xi_1-u}
\end{equation}
which leads to
\begin{equation}\label{eq:est:intlchi12ds}
 \begin{split}
  \int_{-1}^1 \chi_{\sabs{\xi_1-u+s_1w}\le\sabs{\xi_1-u}}m(\xi_1-u,0,s_1,w)\d s_1
  &\aleq \abs w^\delta \int_{\sabs{\xi_1-u+s_1w}\le\sabs{\xi_1-u}} \sabs{\xi_1-u+s_1w}^{-1/2-\delta}\d s_1 \\
  &\aleq \abs {w}^{\delta-1} \int_{\sabs{\sigma}\le\sabs{\xi_1-u}} \sabs{\sigma}^{-1/2-\delta}\d\sigma 
  \aleq \abs {w}^{\delta-1} \sabs{\xi_1-u}^{1/2-\delta} \\
  &\refeq[\aleq]{est:chi12aleqw} \sabs{\xi_1-u}^{-1/2-\delta}\abs {w}^\delta,\qquad\text{for }\delta\in(0,\tfrac12).
 \end{split}
\end{equation}
Here we made use of the fact that, given case~2 for $a:=\xi_1-u$ and~\eqref{eq:x11a},
\begin{align*}
 \sabs{a} &\leq \min \br{\sabs{a+sw} + \sabs{s}\sabs{w}, \sabs{a+tw} + \sabs{t}\sabs{w}}
 \leq \min \br{\sabs{a+sw}, \sabs{a+tw} }+\sabs{w}\\
 &\refeq[\leq]{x11a} \;\tfrac12{\sabs{a}}+\sabs{w},
\end{align*}
implies
\begin{equation}\label{eq:est:chi12aleqw}
 \sabs{a} \leq  2\sabs{w}.
\end{equation}
Applying~\eqref{eq:est:chi1roughest} and~\eqref{eq:est:chi2roughest}, we arrive at
\begin{equation}\label{eq:k-1}
 \begin{split}
 &\int_{[-1,1]^3} {k(\xi,u,s,w)}\d s\\
 &\aleq \abs w^{\delta-2}\sabs{\xi_1-u}^{-1/2-\delta}\iint_{[-1,1]^2}m(\xi_2-u,s_2,s_3,w)m(\xi_3-u,s_2,s_3,w)\d s_2\d s_3 \\
 &\aleq \abs w^{3\delta-2}\sabs{\xi_1-u}^{-1/2-\delta}
  \iint_{[-1,1]^2}\mu_2\br{\sabs{\xi_2-u+s_2w}^{-1/2-\delta},\sabs{\xi_2-u+s_3w}^{-1/2-\delta}} \cdot \\
 &\qquad\qquad\qquad\qquad\qquad\qquad{}\cdot\mu_3\br{\sabs{\xi_3-u+s_2w}^{-1/2-\delta},\sabs{\xi_3-u+s_3w}^{-1/2-\delta}}\d s_2\d s_3
 \end{split}
\end{equation}
where
\[ \mu_i=\mu_i\br{\xi_i-u,s_2,s_3,w}\in\set{\min,\max}, \qquad i=2,3, \]
depending on the respective case.
If case~1 holds for at least one of the two factors in the integrand, say the first one, we may choose the argument of
\[ \mu_2\br{\sabs{\xi_2-u+s_2w}^{-1/2-\delta},\sabs{\xi_2-u+s_3w}^{-1/2-\delta}}\]
which contains the same integration variable
as the second one. This results in terms of type~\eqref{eq:typII}.
If, however, case~2 applies to both factors, the integral in~\eqref{eq:k-1} is bounded by
\begin{align*}
 &\iint_{[-1,1]^2}\br{\sabs{\xi_2-u+s_2w}^{-1/2-\delta}+\sabs{\xi_2-u+s_3w}^{-1/2-\delta}}\cdot\\
 &\qquad\quad{}\cdot\br{\sabs{\xi_3-u+s_2w}^{-1/2-\delta}+\sabs{\xi_3-u+s_3w}^{-1/2-\delta}}\d s_2\d s_3.
\end{align*}
Expanding the integrand, the terms
$\sabs{\xi_2-u+s_iw}^{-1/2-\delta}\sabs{\xi_3-u+s_iw}^{-1/2-\delta}$, $i=2,3$, lead us to~\eqref{eq:typII}.
For the two remaining terms we may separate the integrals which gives
\[ \sum_{\br{\s(2),s(3)}\in\set{(2,3),(3,2)}}
   \int_{-1}^1\sabs{\xi_{\sigma(2)}-u+s_2w}^{-1/2-\delta}\d s_2
   \int_{-1}^1\sabs{\xi_{\sigma(3)}-u+s_3w}^{-1/2-\delta}\d s_3. \]
One integral is kept in order to arrive at~\eqref{eq:typI}, the other one is treated analogously to~\eqref{eq:x11a} and~\eqref{eq:est:intlchi12ds}.
In order to estimate $\Gamma(u)$,
we obtain thus for $\delta_{1},\delta_{2} \in (0,\frac{1}{2})$
\[
\begin{ma}
 &&\intl_{\R^3}\ \intl_{(-1,1)^3} F(\xi_1)\ F(\xi_2)\ F(\xi_3)\ {k(\xi,u,s,w)}\d s \d\xi\\
&\aeq& \intl_{-1}^1 \sabs{w}^{-1+(3\delta-1)} \sabs{\lapms{\frac{1-2\delta}{2}} {F}(u)}^2\ \lapms{\frac{1-2\delta}{2}} {F}(u-tw)\d t\\
&&+ \intl_{-1}^1 \sabs{w}^{-1+(3\delta-1)} \lapms{\frac{1-2\delta}{2}} {F}(u)\ \sabs{\lapms{\frac{1-2\delta}{2}} {F}(u-tw)}^2\d t,
\end{ma}
\]
which implies
\[
\begin{ma}
\Gamma(u) &\aleq&  \intl_{-1}^1 \intl_{\tilde{w}} t^{2-3\delta}\ \sabs{\tilde{w}}^{-1+(3\delta-1)} \sabs{\lapms{\frac{1-2\delta}{2}} {F}(u)}^2\ \lapms{\frac{1-2\delta}{2}} {F}(u-\tilde{w})\  t^{-1} \d\tilde{w}\d t\\
&& + \intl_{-1}^1 \intl_{\tilde{w}}t^{2-3\delta}\ \sabs{\tilde{w}}^{-1+(3\delta-1)} \lapms{\frac{1-2\delta}{2}} {F}(u)\ \sabs{\lapms{\frac{1-2\delta}{2}} {F}(u-\tilde{w})}^2\ t^{-1} \d\tilde{w}\d t\\

&\overset{\delta \in(\frac13, \frac{1}{2})}{\aeq}& \intl_{w} \sabs{w}^{-1+(3\delta-1)} \sabs{\lapms{\frac{1-2\delta}{2}} {F}(u)}^2\ \lapms{\frac{1-2\delta}{2}} {F}(u-w)\d w\\
&& + \intl_{w} \sabs{w}^{-1+(3\delta-1)} \lapms{\frac{1-2\delta}{2}} {F}(u)\ \sabs{\lapms{\frac{1-2\delta}{2}} {F}(u-w)}^2\d w\\

&\aeq&  \sabs{\lapms{\frac{1-2\delta}{2}} {F}(u)}^2\ \lapmsv{4\delta-1}{F}(u)+\lapms{\frac{1-2\delta}{2}} {F}(u)\ \lapms{3\delta-1}\sabs{\lapms{\frac{1-2\delta}{2}} {F}(u)}^2.
\end{ma}
\]
Setting $\delta := \frac{5}{12}$, this is \eqref{eq:Gammauest}.

\subsection{
Proof of Lemma~\ref{la:rhsests}}
Plugging together \eqref{eq:antisympde}, Lemma \ref{lem:CriticalTerm}, and Lemma \ref{la:lhs}
we get for small $s$ and some $\theta >0$ that
\begin{align*}
 &\int \dsolnt_i \omega_{ij}\ \lapv \dsolnt_j\ \lapv \varphi \\
 &\aleq \vrac {\varphi}_{\infty}\ \vrac{\Gamma}_{1,B_r}
+ \vrac{\lapms{s} \lapv \varphi}_{(\frac{2}{1-2s},1)} \cdot\\
&\qquad{}\cdot\brac{r^{\fracm{2}} + 
\|\lapv  \dsolnt\|_{(2,\infty),B_{\Lambda r}}^2 + \norm{\lapv  \dsolnt}_{(2,\infty),\R}
    \sum_{k=2}^\infty \Lambda^{-\theta (k-1)} \|\lapv   \dsolnt\|_{(2,\infty), B_{\Lambda^kr}} 
}.
\end{align*}
By the estimate of $\Gamma$ from Lemma \ref{la:T2est}, for $F := \sabs{\lapv \dsolnt}$,
\[
\begin{ma}
 \vrac{\Gamma}_{1,B_r} &\aleq& \Vrac {\sabs{\lapms{\frac{1}{12}} \sabs{F}}^2\ \lapms{\frac{1}{3}}\sabs{F}+\lapms{\frac{1}{12}} \sabs{F}\ \lapms{\frac{1}{4}}\br{\sabs{\lapms{\frac{1}{12}} \sabs{F}}^2} }_{1,B_r}\\
&\aleq& \Theta\  \brac{\Vrac{\lapms{\frac{1}{3}}\sabs{F}}_{(6,\infty),B_r} + \Vrac{\lapms{\frac{1}{12}}\sabs{F}}_{(\frac{12}{5},\infty),B_r}},
\end{ma}
\]
where
\[
 \Theta := \Vrac{\lapms{\frac{1}{12}} \sabs{F}}_{(\frac{12}{5},2),B_r}^2 + \Vrac{ \lapms{\frac{1}{4}} \br{\sabs{\lapms{\frac{1}{12}} \sabs{F}}^2}  }_{(\frac{12}{7},1),B_r}.
\]
Observe,
\[
\begin{ma}
 \vrac{ \lapms{\frac{1}{3}} {F}}_{(6,\infty)} &\aleq& \Vert F \Vert_{(2,\infty)} \\
\vrac{ \lapms{\frac{1}{12}} {F}}_{(\frac{12}{5},2)} &\aleq& \Vert F \Vert_{2},
\end{ma}
\]
and because $\frac{5}{12} + \frac{5}{12} + \frac{1}{6} = 1$,
\[
\begin{ma}
\vrac{ \lapms{\frac{1}{12}} {F}}_{(\frac{12}{5},q)} &\aleq& \Vert F \Vert_{(2,q)}, \\
\vrac{ \lapms{\frac{1}{4}}\sabs{\lapms{\frac{1}{12}} {F}}^2 }_{(\frac{12}{7},1)} &\aleq& \vrac{\sabs{\lapms{\frac{1}{12}} {F}}^2 }_{(\frac{6}{5},1)} = \vrac{\sabs{\lapms{\frac{1}{12}} {F}} }_{(\frac{12}{5},2)}^2 \aleq \Vert F \Vert_2.
\end{ma}
\]
Consequently, $\Theta$ is uniformly small, if $r$ is small enough. 
In order to conclude, it only remains to apply Lemma~\ref{lem:QuasiLocalityOfNorms} to $f:= \sabs{F}$.\qed

\section{Bootstrapping: Proof of Theorem \ref{thm:bootstrapping}}\label{sec:bootstrapping}
In Theorem \ref{th:reg:1alpha} we have shown, that $\lapv \dsoln \in L^p$ for some $p > 2$. We now work with Bessel-potential / Sobolev spaces $H^{s,q}$, cf. \cite{RS,Tartar2007,Triebel}, and the fact that $\lapv \dsoln \in L^p$ readily implies that $\dsoln \in H^{\frac{3}{2},\tilde{p}}$ for some $\tilde{p} \in (2,p)$. The proof of Theorem~\ref{thm:bootstrapping} relies on the decomposition of the first variation
\[
 \delta E^{(2)}(\g,h) =  2\lim_{\varepsilon \searrow 0} \iint_{U_{\varepsilon}}
\left(
\frac {\left\langle \g'(u), h'(u) \right\rangle}{|\g(u+w)-\g(u)|^2}
-\frac {\left\langle \g(u+w)- \g(u), h(u+w)-h(u) \right\rangle}{|\g(u+w)-\g(u)|^4} 
\right) \d w \d u,
\]
$U_{\varepsilon}$ as in~\eqref{eq:Ue}, into 
\begin{equation*}
2\br{Q(\g,h) - T_1 (\g, h ) - T_2 (\g,h)}.
\end{equation*}

For a stationary point of the M\"obius energy we have
\begin{equation*}
 Q(\g,h)= T(\g,h):= T_1(\g,h) + T_2(\g,h)
\end{equation*}
for all $h \in H^{3/2}(\mathbb R/ \mathbb Z, \mathbb R^n)$. 

Let us bring these terms in a common form. Using
\begin{multline*}
 \frac 1 {|\g(u+w) - \g(u)|^\alpha} - \frac 1 {|w|^\alpha} 
  = \frac {|w|^\alpha}{|\g(u+w)-\g(u)|^\alpha} 
     \left(\frac {1- \frac {|\g(u+w)-\g(u)|^\alpha}{|w|^\alpha}} {|w|^\alpha} \right) \\
  = G^\alpha\left(\frac {\g(u+w)-\g(u)}{w}\right) \left(
     \frac {2-2\frac {|\g(u+w)-\g(u)|^2}{|w|^2}} {|w|^\alpha} \right) \\
  = \int_{0}^1 \int_{0}^1 G^\alpha\left(\frac {\g(u+w)-\g(u)}{w}\right) \left(
  \frac {|\g'(u+\tau_1 w) - \g'(u+\tau_2 w)|^2}{|w|^\alpha}  \right) d\tau_1 d \tau_2
\end{multline*}
where
\begin{align*}
 G^\alpha(z) := \frac 1 {2 |z|^\alpha} \cdot \frac {1- |z|^\alpha}{1 - |z|^2}
\end{align*}
is an analytic function away from the origin for $\alpha\ge 2$.
We hence get
\begin{align*}
  T_1(\g,h) &= -\int_0^1 \int_0^1 T^{2}_{0,0,\tau_1, \tau_2}(h) d\tau_1 d\tau_2 \\
 T_2(\g,h) &= \int_0^1 \int_0^1 \int_0^1 \int_0^1 T^{4}_{s_1,s_2, \tau_1, \tau_2} (h) d\tau_1 d\tau_2 ds_1 ds_2
\end{align*}
where
\begin{multline*}
 T^{\alpha}_{s_1, s_2, \tau_1, \tau_2}(h) \\:= \int_0^1 \int_{-1/2}^{1/2} G^{\alpha} 
\left(\frac {\g(u+w) - \g(u)}{w} \right) \frac {|\g'(u+\tau_1 w) -\g'(u+\tau_2 w)|^2}{w^2} 
\g'(u+s_1 w) h'(u+s_2 w) dw \ du.
\end{multline*}

In the rest of this section, we will derive some estimates for the linear operators $T^\alpha_{s_1,s_2, \tau_1, \tau_2}$
that do not depend on $s_1,s_2, \tau_1,$ and $\tau_2$.

For this task, we will work with the Besov spaces $B^{s}_{p,q}$.
Given $s \in (0,1)$ and $p,q \in [1,\infty)$ one way to define the norm on these spaces 
is to set
\begin{equation*}
 |f|_{B^{s}_{p,q}} := \left(\;\int_{-1/2}^{1/2} 
  \frac {\left(\int_{\mathbb R / \mathbb Z}|f(u+w)-f(u)|^p du \right)^{q/p}}{|w|^{1+sq}} dw\right)^{1/q} 
\end{equation*}
 and to put 
\begin{equation*}
 \|f\|_{B^{s}_{p,q}} := \|f\|_{L^p} + |f|_{B^{s}_{p,q}}.
\end{equation*}
The Besov space $B^{s}_{p,q}(\R / \Z, \R^\tdim)$ then consists of all functions $f\in L^p$ with
$|f|_{B^s_{p,q}} < \infty$.

Apart from this definition we just need the Sobolev embedding
\begin{equation*}
 H^{\tilde s, \tilde p} \subset B^{s}_{p,q} 
\end{equation*}
if $s < \tilde s$ and $s- \frac 1 p < \tilde s - \frac 1 {\tilde p}$ which can be found
in textbooks like, e.g., \cite{Triebel,RS,Tartar2007}.
We also refer to~\cite{DiNezza}.

The proof relies furthermore on the following rules for 
Bessel potential spaces.

\begin{lemma}[Fractional Leibniz rule~\cite{CM}, {\cite[Lem.~5.3.7/1~(i)]{RS}}
] \label{lem:LeibnizRule}
Let $f \in H^{s,p}(\mathbb R / \mathbb Z, \mathbb R^n)$, $g \in H^{s,q}(\mathbb R / \mathbb Z, \mathbb R^n)$
with $s >0$, $p,q,r \in (1,\infty)$ and $1/p + 1/q = 1/r$.
 
Then $fg \in H^{s,r}(\mathbb R / \mathbb Z, \mathbb R^n) $ and
\begin{equation*}
 \|fg\|_{H^{s,r}} \leq C \left( \|f\|_{H^{s,p}} \|g\|_{L^q} + \|f\|_{L^p} \|g\|_{H^{s,q}}\right).
\end{equation*}

\end{lemma}

For the following statement, one mainly has to treat $\norm{(D^k\psi)\circ f}_{\W[\s,p]}$ for $k\in\N\cup\set0$ and $\s\in(0,1)$
which is e.~g.\@ covered by~\cite[Thm.~5.3.6/1~(i)]{RS}.

\begin{lemma}[Fractional chain rule]\label{lem:chain}
 Let $f\in\W[s,p](\mathbb R / \mathbb Z, \mathbb R^n)$, $s > 0$, $p \in (1,\infty)$.
 If $\psi\in C^\infty(\R)$ is globally Lipschitz continuous and $\psi$ and all its derivatives vanish at~$0$ then $\psi\circ f\in\W[s,p]$ and
 \[ \|\psi \circ f\|_{\W[s,p]} \le C\|\psi\|_{C^{\scriptstyle k}}\|f\|_{\W[s,p]} \]
 where $k$ is the smallest integer greater than or equal to $s$.
\end{lemma}

The key to the proof of Theorem~\ref{thm:bootstrapping} is the following lemma.

\begin{lemma} \label{lem:basicEstimate}
Let $\g \in H^{\frac 3 2 + \beta_0,q} (\mathbb R / \mathbb Z, \mathbb R^n)$,
$\alpha\ge2$,
and $\beta_0 \ge \beta \ge 0$, $p,q \in (1, \infty)$ be such that 
$\beta_0 - 1/q > \beta - \frac 1 {2p}$.
Then for all $\tau_1, \tau_2, s_1\in \mathbb [0,1]$ the function
\begin{gather*}
g(u):= \int_{-1/2}^{1/2} G^{\alpha}\left(\frac {\g(u+w) - \g(u)}{w} \right) \frac{|\g'(u+\tau_1 w) - \g'(u+\tau_2 w)|^2}{w^2} \g'(u+s_1 w)
dw 
\end{gather*}
is in $H^{\beta,p}$. Furthermore, there is a constant $C<\infty$ depending on
$\|\g\|_{H^{3/2 + \beta_0 ,q}}$ and $\alpha$, but not on $\tau_1, \tau_2$, and $s_1$, such that
\begin{equation*}
 \|g\|_{H^{\beta,p}} \leq C.
\end{equation*}
\end{lemma}

\begin{proof}
Note that
\begin{equation*}
  \|g\|_{H^{\beta , p}} \leq \int_{-1/2}^{1/2} \frac{\|g_w\|_{H^{\beta,p}}}{|w|^2} dw
\end{equation*}
where 
\begin{equation*}
 g_w(u) := G^{\alpha}\left(\frac {\g(u+w)-\g(u)} {w}\right) \g'(u+s_1 w )  |\g'(u+\tau_1 w)-\g(u+ \tau_2 w)|^2 .
\end{equation*}
Choosing some $\tilde p\in(p,p+1)$ which will be determined later on
and letting $\tilde q := 2\frac{\tilde pp}{\tilde p-p}$ leads to
\[ \frac1p = \frac1{2\tilde p} + \frac1{2\tilde p} + \frac1{\tilde q}
   + \frac1{\tilde q}. \]

Using that
\begin{equation*}
 \frac {\g(u+w)-\g(u)} {w} = \int_{0}^1 \g'(u+\tau w) d\tau,
\end{equation*}
that $\g$ is bi-Lipschitz, and that $G^\alpha$ is analytic away from the origin, 
we get according to the fractional chain rule (Lemma~\ref{lem:chain})
\begin{equation*}
 \left\|G^{\alpha}\left(\frac {\g(\cdot +w)-\g(\cdot)}{w}\right)\right\|_{H^{\beta,\tilde q}}
 \leq C \|\g\|_{H^{\beta+1,\tilde q}}
  \leq C.
\end{equation*}

Using the fractional Leibniz rule (Lemma~\ref{lem:LeibnizRule}), we derive
\begin{align*}
 \|g_w\|_{H^{\beta,p}} &\leq C
\left\|G^{\alpha}\left(\frac {\g(\cdot +w)-\g(\cdot)}{w}\right)\right\|_{H^{\beta,\tilde q}}
\|\g'\|_{H^{\beta,\tilde q}}
\|\g'(\cdot + \tau_1 w) - \g'(\cdot + \tau_2 w )\|^2_{H^{\beta,2\tilde p}}
 \\
&  \leq C \br{\|\g'(\cdot + \tau_1 w) - \g'(\cdot + \tau_2 w )\|^2_{L^{2\tilde p}} 
+ \|\laps{\beta+1}\g(\cdot + \tau_1 w) - \laps{\beta+1}\g(\cdot + \tau_2 w )\|^2_{L^{2\tilde p}} }.
\end{align*}
Hence,
\begin{align*}
  \|g\|_{H^{\beta,p}} &\leq C
    \int_{-1/2}^{1/2} \frac{\|\g'(\cdot + \tau_1 w) - \g'(\cdot + \tau_2 w )\|^2_{L^{2\tilde p}} }{w^2} dw
  \\
  & \quad + C
  \int_{-1/2}^{1/2} 
    \frac {\|\laps{\beta+1}\g(\cdot + \tau_1 w) - \laps{\beta+1}\g(\cdot + \tau_2 w )\|^2_{L^{2\tilde p}}}
  {w^2} dw
\\
&\leq C
\int_{-1/2}^{1/2} \frac{\|\g'(\cdot) - \g'(\cdot + (\tau_2 - \tau_1) w )\|^2_{L^{2\tilde p}} }{w^2} dw
 \\
& \quad + C
  \int_{-1/2}^{1/2} \frac {\|\laps{\beta+1}\g(\cdot) - \laps{\beta+1} \g(\cdot + (\tau_2 - \tau_1) w )\|^2_{L^{2\tilde p}} }{w^2} dw
\\
&\leq C |\tau_2 - \tau_1|
\int_{-1}^{1} \frac{\|\g'(\cdot) - \g'(\cdot +  w )\|^2_{L^{2\tilde p}} }{w^2} dw
 \\
& \quad + C |\tau_2-\tau_1|
  \int_{-1}^{1} \frac {\|\laps{\beta+1}\g(\cdot) - \laps{\beta+1} \g(\cdot +  w )\|^2_{L^{2\tilde p}} }{w^2} dw
\\
& \leq C \|\g'\|_{B^{1/2}_{2\tilde p,2}}^2 
 + C \|\laps{\beta+1}\g\|^2_{B^{1/2}_{2\tilde p,2}} \leq C
\end{align*}
if $\tilde p\in(p,p+1)$ is chosen so small that
\begin{equation*}
  \beta_0 - \frac 1 q > \beta - \frac 1 {2\tilde p} > \beta - \frac 1 {2p}.
\end{equation*}
This proves Lemma~\ref{lem:basicEstimate}.
\end{proof}

We use the last lemma to prove
\begin{corollary} \label{cor:basicEstimate}
 Let $\g, \beta_0, \beta$,$p$ and $q$ be as in Lemma~\ref{lem:basicEstimate} and $p'$ 
 be such that $1/p + 1/p' =1$. Then
\begin{itemize}
 \item for all $\alpha\ge2$ there is a constant $C$
  such that
  \begin{equation*}
   | T^{\alpha}_{s_1, s_2 , \tau_1, \tau_2} (h) | \leq C \|h\|_{H^{1-\beta,p'}}
  \end{equation*}
  for all $s_1, s_2, \tau_1, \tau_2 \in [0,1]$ and $h \in C^\infty$,
\item the operator
 $T(\g,\cdot) = T_1(\g,\cdot) + T_2(\g,\cdot) \in \left(H^{3/2}(\mathbb R / \mathbb Z, \mathbb R^n )\right)^\ast$
 can be extended to a bounded  operator on $ H^{1 - \beta,p'}$.
\end{itemize}
\end{corollary}

\begin{proof}
 For $h \in C^{\infty}(\mathbb R / \mathbb Z, \mathbb R^n)$ we have
 \begin{align*}
  T(\g,h) = - \int_0^1 \int_0^1 T^2_{0,0,t_1,t_2}(h) dt_1 dt_2
+ \int_0^1 \int_0^1 \int_0^1 \int_0^1 T^4_{s_1,s_2,t_1,t_2}(h) ds_1  ds_2 dt_1 dt_2
 \end{align*}
 and hence the second part is an immediate consequence of the first one.

 Let $\Lambda^s := (\id-\Delta)^{\frac s 2}$. Using that $\Lambda^\beta$ is self adjoint, we get
 \begin{align*}
  T^{\alpha}_{s_1,s_2, \tau_1, \tau_2}(h) &= \int_{-1/2}^{1/2} \int_{\mathbb R / \mathbb Z} 
  G^{\alpha}\left(\frac {\g(\cdot +w)-\g(\cdot)}{w}\right)
  \frac {|\g'(u+\tau_1 w ) - \g'(u+\tau_2 w)|^2 }{w^2} 
  \\
  & \quad \quad \quad \quad \quad \quad \quad \quad \quad \quad \quad \quad    \g'(u+s_1 w) h'(u+s_2 w)  du dw
  \\&= \int_{\mathbb R / \mathbb Z}\int_{-1/2}^{1/2}
  \frac{\left(\Lambda^{\beta}g_w \right)(u)}{w^{2}}
\left( \Lambda^{-\beta}h'\right)(u+s_2 w) dw du 
 \end{align*}
 and hence as in the proof of Lemma~\ref{lem:basicEstimate}
 \begin{equation*}
  |T_{s_1,s_2, \tau_1, \tau_2}(h)|
  \leq C \int_{-1/2}^{1/2}\frac{\norm{g_{w}}_{\W[\beta,p]}}{w^{2}}\d w\norm h_{\W[1-\beta,p']}
  \le C \norm h_{\W[1-\beta,p']}
 \end{equation*}
 where $C<\infty$ as in Lemma~\ref{lem:basicEstimate} does not depend on 
 on $s_1,s_2,\tau_1$,or $\tau_2$.
\end{proof}

Using the two statements above, we are led to the following fact from which Theorem~\ref{thm:bootstrapping}
immediately follows.

\begin{lemma}\label{lem:bootstrappingStep}
 Let $\g \in H^{3/2 + \beta_0,q}(\mathbb R / \mathbb Z, \mathbb R^n)$, $\beta_0 \geq 0$, $q\in[2,\infty)$ and
 $\beta_0 - \frac 1 q > -1/2$ 
  be a stationary point of the M\"obius energy parametrized by
  arc length. Then,
\begin{itemize}
 \item if $\beta_0 =0$, we have $\g \in H^{s}$ for all $s < 3/2 + 2(1/2-1/q)$,
 \item if $0< \beta_0 < 1/2$, we have $\g \in H^{s}$ for all $s < 3/2 + 2(\beta_0 + 1/2-1/q)$,
 \item if $\beta_0 \geq 1/2$, we have $\g \in H^{3/2+\beta_0+1/4}$.
\end{itemize}

\end{lemma}

\begin{proof}
 We set
\begin{align*}
  \beta & =0, & \frac 1p &= \frac 2 q + \varepsilon, & & \text { if } \beta_0 =0, \\
  \beta &=0,& \frac 1 p &= \frac 2 q - 2\beta_0 + \varepsilon,& &\text{ if } 0< \beta_0 < 1/2, \\
  \beta &= \beta_0 - 1/4,& \frac 1 p&=2/3, & & \text{ if } \beta_0 \geq \frac 1 2
\end{align*}
 and see that in each case the exponents satisfy the assumptions for the
 preceding Corollary~\ref{cor:basicEstimate}
 for all small enough $\varepsilon >0$, so,
 for $p'$ with $\frac 1 p + \frac 1 {p'} =1$, the functional
 $T(\g,\cdot) = T_1(\g,\cdot) + T_2(\g,\cdot)$
 can be extended to an operator in 
 $\left(H^{1-\beta, p'}\right)^{\ast} \subset (H^{3/2-1/{p'}-\beta})^\ast$.
 
From the fact that $\g$ is a stationary point of the M\"obius energy we then deduce that
 \begin{equation*}
    Q(\g,\cdot) \in  (H^{3/2-1/{p'}-\beta})^\ast
 \end{equation*}
 and a comparison of the Fourier coefficients gives
 \begin{equation*}
  \g \in H^{3/2+\beta + \frac 1 {p'}}.
 \end{equation*}
Since
\begin{equation*}
 3/2+ \beta + \frac 1 {p'} = 5/2 + \beta - \frac 1 p  = 
\begin{cases} 
  3/2 + 2 (1/2-1/q)-\varepsilon &\text{if } \beta_0=0, \\
  3/2 + 2(\beta_0 + 1/2-1/q)  - \varepsilon &\text{if } 0< \beta_0 < 1 /2, \\
  3/2 + \beta_0 + 1/4 &\text{else},
\end{cases}
\end{equation*}
this proves Lemma~\ref{lem:bootstrappingStep}.
\end{proof}


\begin{appendix}
\section{Appendix}

In this section we gather some facts most of which can already be found in 
\cite{SchikorraIntBoundFrac11} in slightly different versions. The main aim is to prove 
Lemma~\ref{la:tangentialpart} and \ref{lem:CriticalTerm} which both rely on 
quasi-locality of the Riesz potential $\lapms{s}$.  Afterwards, we give an easy proof of the iteration lemma needed
to deduce Dirichlet growth. 

We will mainly deal with functions belonging to the Schwartz space $\mathcal{S}(\R)$
of rapidly decreasing smooth functions $\R\to\R$. The statements carry over to more general situations by suitable approximation arguments.

\subsection{Quasi-locality}

The essential tool apart from Sobolev inequalities is the following 
quantitative version of the \emph{quasi-locality} of the fractional Laplacian 
and the Riesz potential.

\begin{lemma}[Quasi-locality]\label{lem:QuasiLocality}
 Let $p_1,p_2,q_1,q_2 \in [1, \infty]$, where we assume that $q_2 =1$ if $p_2=1$, $s \in (-1,1)$
 and $\Omega_1, \Omega_2$ be disjoint domains with
 $d:= \dist(\Omega_1, \Omega_2) >0$ and with positive and finite Lebesgue measure. 
 Then, for any $f\in\mathcal S(\R)$, 
 \begin{equation*}
  \|\lap^{\frac{s}{2}} (f \chi_{\Omega_2})\|_{(p_1,q_1),\Omega_1} 
  \aleq  d^{-1-s} |\Omega_1|^{1/p_1} |\Omega_2|^{1-1/p_2} \|f\|_{(p_2,q_2), \Omega_2}
 \end{equation*}
  where we set
 \[
  \lap^{\frac{s}{2}} := \begin{cases}
                         \laps{s} \quad &\mbox{if $s > 0$,}\\
			  \id \quad &\mbox{if $s = 0$,}\\
			  \lapms{\abs{s}} \quad &\mbox{if $s < 0$.}\\
                        \end{cases}
 \]
\end{lemma}

Note, that $\|\cdot\|_{(1,q_1)}$, $\|\cdot\|_{(\infty,q_\infty)}$ are only considered
in the inequalities that follow for $q_1=1$ and $q_\infty=\infty$.

\begin{proof}
 For $\tilde k_s (z) = \frac 1 {|z|^{1+s}} \chi_{\R\setminus B_d} (z)$
 and $\supp f\subset\Omega_{2}$ we have
 for all $x \in \Omega_1$
 \begin{equation*}
  \lap^{\frac{s}{2}} f(x) = c_{s}(\tilde k_s \ast f ) (x)
 \end{equation*}
 and hence
 \begin{equation*}
 |\lap^{\frac{s}{2}} f (x)| \aleq \|f\|_{1,\Omega_2} \|\tilde k_s\|_{\infty} \leq
  d^{-1-s} \|f\|_{1,\Omega_2}.
  \end{equation*}
 Hence,
 \begin{align*}
  \|\lap^{\frac{s}{2}} f\|_{(p_1,q_1), \Omega_1} 
  & \leq |\Omega_1|^{1/p_1} \|\lap^{\frac{s}{2}} f\|_{\infty, \Omega_1}
  \leq d^{-1-s}|\Omega_1|^{1/p_1}\|f\|_{1, \Omega_2} \\ 
   & \leq d^{-1-s} |\Omega_1|^{1/p_1} |\Omega_2|^{1-1/p_2} \|f\|_{(p_2,q_2), \Omega_2}. \qedhere   
 \end{align*}
\end{proof}

A quite immediate consequence of this quasi-locality and Sobolev 
imbeddings is the following lemma. To state it, let
\begin{equation}\label{eq:annulus}
 A^k_{\Lambda,r}:= B_{2^{k}\Lambda r}- B_{2^{k-1} \Lambda r}.
\end{equation}

\begin{lemma}\label{lem:QuasiLocalityOfNorms}
Let $p \in (1,\infty), q,\tilde q \in [1,\infty]$, $s \in (-1,1/p)$ and $- \frac 1 {p^\ast} = s - \frac 1 p$.
\begin{enumerate}
  \item For $f \in\mathcal S(\R)$ with $\supp f \subset B_r$ we have
    \begin{equation*}
  \|\lap^{-\frac{s}{2}}f\|_{(p^\ast,q);\mathbb R - B_{\Lambda r}} 
  \aleq \Lambda^{-1+\fracm p} \|f\|_{(p, \tilde q),B_{r}}
 \end{equation*}
 uniformly for all $\Lambda >2$.
 
 \item If $s \in[0,\frac1p)$, we have for all $f \in\mathcal S(\R)$
 \begin{equation*}
  \|\lapms{s}f\|_{(p^\ast,q);\mathbb R - B_{\Lambda r}} 
  \aleq \Lambda^{-1+\fracm p} \|f\|_{(p, \tilde q);B_r} + \|f\|_{(p, q);\mathbb R - B_{r}}
 \end{equation*} 
 uniformly for all $\Lambda >2$.

 \item For $s\in[0,\frac1p)$ and any $f\in\mathcal S(\R)$, we have
 \begin{equation*}
    \|\lapms{s}f\|_{(p^\ast,q),B_r} 
    \aleq \|f\|_{(p,q),B_{\Lambda r}} + \Lambda^{-\frac 1{p^\ast}}
	\sum_{k=1}^\infty \left(2^{-k}\right)^{\frac 1 {p^\ast}}
	\|f\|_{(p,\tilde q),A^k_{\Lambda, r}}
 \end{equation*}
 uniformly for all $\Lambda >2$.
\end{enumerate}

\end{lemma}


\begin{proof}
 For the first inequality we use Lemma~\ref{lem:QuasiLocality}
 and sum up the estimate
 \begin{equation*}
  \|\lap^{-\frac{s}{2}} f\|_{(p^*,  q ),A^k_{\Lambda, r}} \aleq (\Lambda 2^k)^{-1+\fracm p} \|f\|_{(p, \tilde q);B_r}.
 \end{equation*}

 In the second case, we use 
 \begin{align*}
  \|I_sf\|_{(p^\ast, q), \mathbb R - B_{\Lambda r}}
  &\leq \|I_s(\chi_{B_r}f)\|_{(p^\ast, q), \mathbb R - B_{\Lambda r}}
  + \|I_s((1-\chi_{B_r})f)\|_{(p^\ast, q), \mathbb R - B_{\Lambda r}}
 \end{align*}
 and estimate the first term using~(i)
 and the second term
 using Sobolev's inequality to get
 \begin{equation*}
  \|I_sf\|_{(p^\ast, q), \mathbb R - B_{\Lambda r}} 
  \aleq \Lambda^{-1+\frac 1 p} \|f\|_{(p,\tilde q), B_r} 
  + \|f\|_{(p,q); \mathbb R - B_{r}}.
 \end{equation*}

 To deduce the last inequality, we decompose $f= \chi_{B_{\Lambda r}} f + \sum_{k=1}^\infty \chi_{A^{k}_{\Lambda, r}} f$ 
 and estimate using Sobolev's inequality
 \begin{align*}
  \|\lapms{s}(\chi_{B_{\Lambda r}}f)\|_{(p^\ast,q),B_r} \aleq \|f\|_{(p,q),B_{\Lambda r}}
 \end{align*}
 and using Lemma~\ref{lem:QuasiLocality}
 \begin{align*}
  \|\lapms{s}(\chi_{A^{k}_{\Lambda,r}}f)\|_{(p^\ast,q), B_r} 
  \aleq (\Lambda 2^k)^{-\frac 1 {p^\ast}}\|f\|_{(p,\tilde q),A^k_{\Lambda,r}}.
 \end{align*}
 Summing up, this proves the last inequality.
\end{proof}


Finally, we use the quasi-locality to prove

\begin{proposition}\label{pr:localcontrolelliptic}
For $p \in (1,\infty)$, $q \in [1,\infty]$, $s,t \geq 0$ with $0< s+t < 1$
there is a $\theta >0$ such that
we have for any $f \in L^{p,q}(\R)$, $\Lambda > 2$, and $r>0$  
\begin{multline*}
 \vrac{\laps{s} f}_{(p,q), B_r}  \aleq 
  \sup_{\ontop{\varphi \in C_0^\infty(B_{\Lambda^2 r}),}{\|\laps{t}\varphi\|_{(p',q')} \leq 1}}  \int f \laps{s+t} \varphi 
  +  r^{-s}\Lambda^{-\theta}\ \|f\|_{(p,q);B_{\Lambda r}} + r^{-s} 
   \Lambda^{-\theta} \sum_{l=1}^\infty 2^{-\theta l} \vrac{f}_{(p,q),A_{\Lambda, r}^l}.
\end{multline*}
\end{proposition}

\begin{proof}
Assume that 
\[
 \sup_{\varphi \in C_0^\infty(B_{\Lambda^{2} r}), \|\laps{t}\varphi\|_{(p',q')} \leq 1}
  \int f \laps{s+t} \varphi  \leq K.
\]
For $g \in C_0^{\infty}(B_r)$ we consider
\begin{align*}
 \int_{\mathbb R} (\laps{s} f) g dx= \int_{\mathbb R} f\laps{s} g dx
 = \int_{B_{\Lambda r}} f\laps{s} g dx+ \sum_{l=1}^\infty \int_{A^l_{\Lambda, r}} f\laps{s} g dx.
\end{align*}
For the first term we use a smooth partition of unity $\br{\eta_k}_{{k\in\N\cup\set0}}$
with $\supp \eta_0 \subset B_{\Lambda^2 r}$ and 
$\supp \eta_k \subset B_{(\Lambda^{2} 2^{k}+1)r} -B_{(\Lambda^{2} 2^{k-1}-1)r}$
to get
\begin{align*}
   \left|\int_{B_{\Lambda r}} f\laps{s} g dx\right| 
   &\leq \left|\int_{B_{\Lambda r}} f\laps{t+s} \br{\eta_0
  \lapms{t} g} dx\right|+ \sum_{l=1}^\infty
  \left|\int_{B_{\Lambda r}} f\laps{t+s} \br{\eta_l
  \lapms{t} g} dx\right| \\
  &\aleq K \| g\|_{(p',q')} + \|f\|_{(p,q); B_{\Lambda r}} \sum_{l=1}^\infty 
   \|\laps{t+s} \br{\eta_l
  \lapms{t} g}\|_{(p',q');B_{\Lambda r}}.
\end{align*}
Since by Lemma~\ref{lem:QuasiLocality} we have
\begin{align*}
 \|\laps{t+s} \eta_l \lapms{t} g\|_{(p',q');B_{\Lambda r}} 
 &\aleq (\Lambda^2 2^l r)^{-1-t-s} (\Lambda r)^{1/p'} (\Lambda^2 2^l r)^{1-1/p'}
 \|\eta_l\lapms{t} g\|_{(p',q')}  \\
 &\leq  (2^l )^{-t-s-1/p'} \Lambda^{-(2s+2t+1/p')} r^{-s-t}
 \|\lapms{t}g\|_{(p',q')}  
\end{align*}
we can estimate this further by
\[  K \|g\|_{(p',q')} + \Lambda^{-\theta_{1}}r^{-s-t}\|f\|_{(p,q); B_{\Lambda r}} \|\lapms{t} g\|_{(p',q')} \]
for $\theta_1 = 2t+2s+1/p'$.

Since $\supp g \subset B_r$ we get a scaled Poincar\'e inequality by applying first the Sobolev and then the H\"older inequality
\begin{equation*}
 \|\lapms{t} g\|_{(p',q')} \aleq \norm{g}_{(\frac{p'}{p't+1},q'),B_r}
 \aleq  r^{t} \|g\|_{(p',q')}  
\end{equation*}
and using the quasi-locality (Lemma~\ref{lem:QuasiLocality}) once more
\begin{equation*}
 \|\laps{s} g\|_{(p',q'),A^k_{\Lambda r}}
  \aleq (2^{k}\Lambda r)^{-1-s}r^{1-1/p'}(2^{k}\Lambda r)^{1/p'}\|g\|_{(p',q')}
  = \Lambda^{-\theta_2} 2^{-k \theta_2} r^{-s} \|g\|_{(p',q')}
\end{equation*}
for $\theta_2 = 1+s-1/p'$.

Hence, 
\begin{equation*}
 \int_{\mathbb R} \br{\laps{s} f}\ g dx \leq \left(K + \Lambda ^{-\theta} r^{-s} 
 \|f\|_{(p,q);B_{\Lambda r}} + r^{-s} \Lambda^{-\theta}
  \sum_{l=1}^\infty 2^{-\theta l} 
  \|f\|_{(p,q);A^l_{\Lambda, r}} \right) \|g\|_{(p',q')},
\end{equation*}
for $\theta = \min\{\theta_1, \theta_2\}$ which by duality proves the proposition.
\end{proof}

\subsection{Proofs of Lemmata~\ref{la:tangentialpart} and~\ref{lem:CriticalTerm}}

The following lemma is the starting point for the 
estimates of $H_s$ and essentially follows from the mean value theorem or a first-order Taylor expansion.

\begin{lemma} \label{lem:MultiplierEstimate}
Let $\delta \in [0,1]$, $\alpha \in (0,1)$. Then for almost all $x,y, \xi \in \mathbb R^n$
we have
\begin{equation*}
 \left||x-\xi|^{-1+\alpha} - |y-\xi|^{-1+\alpha} \right| \aleq |x-y|^\delta 
 \left( |y-\xi|^{-1+\alpha-\delta} + |x-\xi|^{-1+\alpha-\delta} \chi_{|x-y|>2|x-\xi|}\right).
\end{equation*}
\end{lemma}

\begin{proof}
 If $|x-y|>2|x-\xi|$ we get
 \begin{equation*}
  |y-\xi| \geq |y-x|-|x-\xi| 
  > |x-\xi|
 \end{equation*}
 and hence
 \begin{equation*}
  \left||x-\xi|^{-1+\alpha} - |y-\xi|^{-1+\alpha} \right| \aleq |x-\xi|^{-1+\alpha}
  \aleq |x-y|^{\delta} |x-\xi|^{-1+\alpha-\delta}.
 \end{equation*}

 If $|x-y|\leq 2 |x-\xi|$ we first observe that the above argument leads to
 \begin{equation*}
   \left||x-\xi|^{-1+\alpha} - |y-\xi|^{-1+\alpha} \right|
  \aleq |x-y|^{\delta} |y-\xi|^{-1+\alpha-\delta}
 \end{equation*}
 if $|x-y|>2|y-\xi|$.

 To deal with the case that both $|x-y|\leq 2 |x-\xi|$ and $|x-y|\leq 2 |y-\xi|$
 we observe that then
 \begin{equation*}
  |y-\xi| \leq |x-\xi| + |x-y| \leq 3 |x-\xi|.
 \end{equation*}
 Hence, we get using the mean value theorem
 \begin{align*}
   \left||x-\xi|^{-1+\alpha} - |y-\xi|^{-1+\alpha} \right|
   &\aleq |x-y| \max\left\{|x-\xi|^{-2+\alpha}, |y-\xi|^{-2+\alpha}\right\}  \\
   &\aleq |x-y| |y-\xi|^{-2+\alpha} \aleq |x-y|^\delta |y-\xi|^{-1+\alpha-\delta}.\qedhere
 \end{align*}
\end{proof}

We use the lemma above to derive the following pointwise estimate for $H_{1/2+s}$.

\begin{lemma}[\cite{SchikorraIntBoundFrac11}]\label{lem:EstimateForH12s}
For $s \in [0,\frac12)$ 
and functions $a,b\in\mathcal S(\R)$
the following holds for any $\varepsilon,\eps' \in [0,\frac16-\frac s3)$:
\begin{align*}
 |H_{1/2+s}(a,b)| &\aleq \lapms{1/6-s/3} (\lapms{1/6-s/3-\varepsilon}|\laps{1/2-\varepsilon} a| \; \lapms{1/6-s/3-\varepsilon'}|\laps{1/2-\varepsilon'} b|) 
   \\ &\quad+ \lapms{1/3-2s/3-\varepsilon}|\laps{1/2-\varepsilon} a| \; \lapms{1/6-s/3-\varepsilon'}|\laps{1/2-\varepsilon'} b|
   \\ & \quad+\lapms{1/6-s/3-\varepsilon}|\laps{1/2-\varepsilon} a|\; \lapms{1/3-2s/3-\varepsilon'}|\laps{1/2-\varepsilon'} b|
   \\ & \quad+\lapms{1/4-s/2-\varepsilon}|\laps{1/2-\varepsilon} a| \; \lapms{1/4-s/2-\varepsilon'}|\laps{1/2-\varepsilon'} b|. 
\end{align*}
\end{lemma}

\begin{proof}
In order to shorten notation, we restrict to $\eps'=\eps$;
the general case is parallel.
 We use the identities $a=\lapms{\fracm{2}-\varepsilon}  \laps{\fracm{2}-\varepsilon}a$ and $b = \lapms{\fracm{2}-\varepsilon}  \laps{\fracm{2}-\varepsilon} b$.
 Then,
\begin{align*}
 &|H_{1/2+s}(a,b)(x)| \\&\aleq \int \int \int \frac {\left(|y-z|^{-1+1/2-\eps}-|x-z|^{-1+1/2-\eps} \right) 
  \left( |y-w|^{-1+1/2-\eps} - |x-w|^{-1+1/2-\eps} \right)}{|y-x|^{1+ (1/2+s)}}\cdot{}\\
  &\qquad\qquad\qquad{}\cdot\abs{\laps{\frac{1-2\eps}2} a(z)} \,\abs{\laps{\frac{1-2\eps}2} b(w)} \d w \d z \d y .
\end{align*}
Applying Lemma~\ref{lem:MultiplierEstimate} we get
\begin{align*}
 & \frac {\left(|y-z|^{-1+1/2-\eps}-|x-z|^{-1+1/2-\eps} \right) 
  \left( |y-w|^{-1+1/2-\eps} - |x-w|^{-1+1/2-\eps} \right)}{|y-x|^{1+ (1/2+s)}} \\
 &\aleq \frac {\scriptstyle\left(|y-w|^{-1+1/2-\eps-\delta} + |x-w|^{-1+1/2-\eps-\delta} \chi_{|x-y|>2|x-w|} \right) 
		\left(|y-z|^{-1+1/2-\eps-\delta} + |x-z|^{-1+1/2-\eps-\delta} \chi_{|x-y|>2|x-z|} \right)}
	     {|x-y|^{1+1/2+s-2\delta}} \\
 &\aleq \frac {(|y-w||y-z|)^{-1+1/2-\eps-\delta} + (|x-w||y-z|)^{-1+1/2-\eps-\delta}  
  + (|y-w||x-z|)^{-1+1/2-\eps-\delta}}{|x-y|^{1+1/2+s-2\delta}} \\
 & \quad{} + \frac {(|x-w||x-z|)^{-1+1/2-\eps-\delta}}{|x-y|^{1+1/2+s-2\delta}}
    \chi_{|x-y|>2|x-w|} \ \chi_{|x-y|>2|x-z|}.
\end{align*}
For $\delta =1/3 + s/3$ we hence get
\begin{align*}
 |H_{1/2+s}(a,b)|& \aleq \lapms{1/6-s/3} (\lapms{1/6-s/3-\eps}|\laps{\frac{1-2\eps}2} a| \ \lapms{1/6-s/3-\eps}|\laps{\frac{1-2\eps}2} b|) 
  \\ & \quad+\lapms{1/6-s/3-\eps}|\laps{\frac{1-2\eps}2} a|  \ \lapms{1/3-2s/3-\eps}|\laps{\frac{1-2\eps}2} b|
  \\ &\quad+ \lapms{1/3-2s/3-\eps}|\laps{\frac{1-2\eps}2} a| \ \lapms{1/6-s/3-\eps}|\laps{\frac{1-2\eps}2} b|
  + A
\end{align*}
where
\begin{align*}
 A &\aleq \int \int \brac{\int \frac {(|x-w||x-z|)^{-1+1/2-\eps-\delta}}{|x-y|^{1+1/2+s-2\delta}}
    \chi_{|x-y|>2|x-w|} \  \chi_{|x-y|>2|x-z|}\ \d y} \	 |\laps{\frac{1-2\eps}2}  a(z)|\ |\laps{\frac{1-2\eps}2}  b(w)| \d w \d z \\
   &\aleq \int \int   |x-w|^{-1+1/4-\eps-s/2 } |x-z|^{-1+1/4-\eps-s/2} |\laps{\frac{1-2\eps}2}  a(z)|\ |\laps{\frac{1-2\eps}2}  b(w)| \d w \d z \\
   &\aleq \lapms{1/4-s/2-\eps}|\laps{\frac{1-2\eps}2} a|\ \lapms{1/4-s/2-\eps}| \laps{\frac{1-2\eps}2} b|. \qedhere
\end{align*}
\end{proof}

To estimate this further, we will use the following fact about {\em lower order products}
which we get using the quasi-locality.

\begin{lemma} [Lower order products] \label{lem:LowerOrderProducts.v2}
  Let for $s > 0$, $0 \leq s_{1}, s_{2}, s_{3} \leq 1/2$, $s_{1} + s_{2}+ s_{3} = s$, and at least two of these $s_i$, $i=1,2,3$, non-zero. Let $p \in (1,\infty)$, $p_2 \in (1,\fracm{s_2}),p_3 \in (1,\fracm{s_3})$ and such that
  \[
   \fracm{p} = \fracm{p_2} + \fracm{p_3} - s.
  \]
Assume moreover that
\begin{equation*}
  |G(u,v)| \aleq \lapms{s_{1}} \left( \lapms{s_{2}}  |u|\ 
  \lapms{s_{3}} |v|\right),
 \end{equation*}
 then there is some $\theta>0$ such that
 we have for any $\Lambda>4$, and $\fracm{q} = \fracm{q_1}+ \fracm{q_2}$,
 \begin{align}\label{eq:Guvest}
    &\|G(u,v)\|_{(p, q), B_r} \\ \nonumber
  &\aleq \|u\|_{(p_2,q_1),B_{\Lambda r}} \|v\|_{(p_3,q_2),B_{\Lambda r}} + \|u\|_{(p_2,q_1)}
    \Lambda^{-\theta}\sum_{k=2}^\infty 2^{-\theta k}\|v\|_{(p_3,q_2),B_{\Lambda 2^k r}}.
 \end{align}
 If $\supp \laps{t}v  \subset \overline{B_r}$ for some $t \in [0,\fracm{2}]$, 
 we furthermore get for any $k \geq 2$, $\Lambda > 16$,
 \begin{equation}\label{eq:Ginversedssupport}
  \|G(u,v)\|_{(p,q),A^k_{\Lambda,r}} 
  \aleq \Lambda^{-\theta} 2^{-\theta k} \vrac{u}_{(p_2,\infty),\R}\ \|v\|_{(p_3, 1),\R}.
 \end{equation}
\end{lemma}
\begin{proof}
Let
\[
  \fracm{p_2^\ast} :=  \fracm{p_2} -s_2, 
\]
\[
  \fracm{p_3^\ast} :=  \fracm{p_3} -s_3, 
 \]
\[
  \fracm{p_1} :=  \fracm{p_2^\ast} + \fracm{p_3^\ast} = \fracm{p_2} + \fracm{p_3} - s_2 - s_3,
\]
and
\[
 \fracm{p_1^\ast} := \fracm{p} \equiv \fracm{p_2} + \fracm{p_3} - s
 =\frac{1}{p_{1}}-s_{1}.
\]

Then by Lemma~\ref{lem:QuasiLocalityOfNorms}~(iii), for $\Lambda \hat = \sqrt{\Lambda}$, $p^\ast \hat = p_1^\ast=p$, $p \hat= p_1$, $s\hat=s_{1}$
 \begin{align*}
  &\|G(u,v)\|_{(p_1^\ast,q),B_r} \aleq
  \vrac{\lapms{s_2}\abs{u} \ \lapms{s_3}\abs{v}}_{(p_1,q), B_{\Lambda^{1/2} r}} + \Lambda^{-1/2p}
  \sum_{k=1}^\infty 2 ^{-k/p} 
  \vrac{\lapms{s_2}\abs{u} \ \lapms{s_3}\abs{v}}_{(p_1,q), B_{\Lambda^{1/2}2^k r}} \\
  &\aleq
 \vrac{\lapms{s_2}\abs{u}}_{(p_2^\ast,q_1), B_{\Lambda^{1/2} r}} \ \vrac{\lapms{s_3}\abs{v}}_{(p_3^\ast,q_2), B_{\Lambda^{1/2} r}}+ \Lambda^{-1/2p}
 \sum_{k=1}^\infty 2 ^{-k/p} 
\vrac{\lapms{s_2}\abs{u}}_{(p_2^\ast,q_1), B_{2^k\Lambda^{1/2} r}} \ \vrac{\lapms{s_3}\abs{v}}_{(p_3^\ast,q_2), B_{2^k\Lambda^{1/2} r}}.
 \end{align*}
Applying yet again Lemma~\ref{lem:QuasiLocalityOfNorms}~(iii) to both factors in each sum with $\Lambda \hat = \sqrt{\Lambda}$
we arrive at
\begin{align*}
    \|G(u,v)\|_{(p, q), B_r} \aleq \|u\|_{(p_2,q_1),B_{\Lambda r}} \|v\|_{(p_3,q_2),B_{\Lambda r}} + \|u\|_{(p_2,q_1)}
    \Lambda^{-\theta}\sum_{k=2}^\infty 2^{-\theta k}\|v\|_{(p_3,q_2),B_{\Lambda 2^k r}}.
\end{align*}

As for the second estimate, we apply Lemma~\ref{lem:QuasiLocalityOfNorms}~(ii)
with $\Lambda \hat = \sqrt{\Lambda}$ to get
\begin{align*}
 \|G(u,v)\|_{(p,q), \mathbb R - B_{\Lambda r}} 
 &=\|\lapms{s_1}\left( \lapms{s_2} |u| \ \lapms{s_3} |v| \right)\|_{(p_1^\ast,q),\mathbb R - B_{\Lambda r}}
 \\
 &\aleq \Lambda^{-\theta_1} 
 \|\lapms{s_2} |u| \ \lapms{s_3} |v|\|_{(p_1, \infty), B_{\Lambda^{1/2} r}}
    + \|\lapms{s_2} |u| \ \lapms{s_3} |v|\|_{(p_1,q), \mathbb R -B_{\Lambda^{1/2}r}}
\end{align*}
where $\theta_{1} = \frac 1 2 - \fracm{2p_1}$.
We estimate the first term using H\"older's inequality and Sobolev imbedding theorem to get
\begin{equation*}
 \Lambda^{-\theta_1} 
 \|\lapms{s_2} |u| \ \lapms{s_3} |v|\|_{(p_1, \infty), B_{\Lambda^{1/2} r}}
 \leq \Lambda^{-\theta_1}\|u\|_{(p_2,\infty)} \|v\|_{(p_3,1)}.
\end{equation*}
For the second term we use H\"older's inequality and Sobolev's imbedding theorem to get
\begin{align*} 
 \|\lapms{s_2} |u| \ \lapms{s_3} |v|\|_{(p_1, q), \mathbb R - B_{\Lambda^{1/2} r}}
 &\aleq \|\lapms{s_2}\abs u\|_{(p^\ast_2,q_1), \mathbb R - B_{\Lambda^{1/2}r}} 
 \|\lapms{s_3}\abs v\|_{(p^\ast_3,q_2), \mathbb R - B_{\Lambda^{1/2}r}} \\
 &\aleq \|u\|_{(p_2,\infty), \mathbb R } 
 \|\lapms{s_3}|v|\|_{(p^\ast_3,q_{2}), \mathbb R - B_{\Lambda^{1/2}r}}.
\end{align*}
Since
\begin{align*}
 \|\lapms{s_3} |v|\|_{(p^\ast_3, q_2), \mathbb R - B_{\Lambda^{1/2} r}}
 & \aleq \Lambda^{-\theta_2} \|v\|_{(p_3,1);B_{\Lambda^{1/4}}} + \|v\|_{(p_3,q);\mathbb R - B_{\Lambda^{1/4}}}
\end{align*}
where $\theta_2 =  \fracm 4 - \fracm {4p_3}$ and,
using Lemma~\ref{lem:QuasiLocalityOfNorms}~(i)
\begin{equation*}
\|v\|_{(p_3,q);\mathbb R - B_{\Lambda^{1/4}}}
= \|\laps{t} (\lapms{t}v)\|_{(p_3,q);\mathbb R - B_{\Lambda^{1/4}}}
\leq \Lambda^{-\theta_2} \|I_t v\|_{(p^{\ast \ast}_3,1)} 
\leq \Lambda^{-\theta_2} \|I_t v\|_{(p_3,1)}
\end{equation*}
we deduce the statement for $\theta := \min\{\theta_1, \theta_2\}$.
\end{proof}

We then have the following
\begin{lemma}\label{lem:localizedHEst}
There is a $\varepsilon_{0} > 0$ such that
for $\delta,\eps\in[0,\varepsilon_{0})$,
for any $a,b\in\mathcal S(\R)$, $\Lambda>4$
\begin{align*}
    &\|H_{\frac{1}{2}}(a,b)\|_{(2, q), B_r} 
  \aleq \|\laps{1/2-\delta} a\|_{(\frac2{1-2\delta},q_1),B_{\Lambda r}} \|\laps{1/2-\eps} b\|_{(\frac2{1-2\eps},q_2),B_{\Lambda r}} + \\
  &\qquad\qquad\qquad\qquad{}+\|\laps{1/2-\delta} a\|_{(\frac2{1-2\delta},q_1)} \Lambda^{-\theta}
    \sum_{k=1}^\infty 2^{-\theta(k-1)}\|\laps{1/2-\eps} b\|_{(\frac2{1-2\eps},q_2),B_{2^k\Lambda r}}.
 \end{align*}
If $\supp b \subset B_r$ we furthermore get for any $k \geq 2$, $\Lambda > 16$, $s\in[0,\frac12)$
 \begin{equation*}
  \|H_{\fracm{2}}(a,b)\|_{(\frac2{1+2s},q),A^k_{\Lambda,r}} \aleq (\Lambda 2^k)^{-\theta} \vrac{\lapv a}_{(2,\infty),\R}\ \|\laps{1/2-\eps} b \|_{(\frac 2 {1-2\eps}, 1),\R}.
 \end{equation*}
 \end{lemma}
 
\begin{proof}
Immediately from Lemma~\ref{lem:EstimateForH12s}, Lemma~\ref{lem:LowerOrderProducts.v2} where $s = \fracm{2} - \varepsilon - \delta$. 
\end{proof}

We use the lemma above to estimate 
the normal part as stated in Lemma~\ref{la:tangentialpart}. 

\begin{proof}[Proof of Lemma \ref{la:tangentialpart}]
With~\eqref{eq:Dg2} we obtain $\norm{\laps{s}\sabs{g'}^2}_{(\frac2{1+2s},\infty),B_{r}}\aleq r^{\frac12+s}$.
In order to show~\eqref{eq:orthogpartest} by the decomposition~\eqref{eq:orthogpartHdec} it remains to treat the $H$-term.
We rewrite
\[
\begin{ma}
 \laps{s} H_{\frac{1}{2}}(a,b) 
&=& \underbrace{H_{s+\fracm{2}}(a,b)}_{=:I}
+ \underbrace{a\lapsv{1+2s}b - \laps{s}(a \lapv b)}_{=:II}
+ \underbrace{b\lapsv{1+2s}a - \laps{s}(b \lapv a)}_{=:III}.
\end{ma}
\]
We will show that all three terms satisfy the hypotheses of Lemma~\ref{lem:LowerOrderProducts.v2}.
Due to Lemma~\ref{lem:EstimateForH12s} this is true for the term $I$.

As for $II$ note that $II=0$ in case $s=0$. If $s\in(0,1)$,
the potential definition~\eqref{eq:fraclap-} of $\laps{s}$ gives
\begin{align*}
 {II(x)} &= c\br{a(x)\int \frac{\lapv b(y)-\lapv b(x)}{\abs{x-y}^{1+s}}\ \d y - \int \frac{a(y)\lapv b(y)-a(x)\lapv b(x)}{\abs{x-y}^{1+s}}\ \d y} \\
 &= c\int \frac{\br{a(x)-a(y)}\lapv b(y)}{\abs{x-y}^{1+s}}\ \d y.
\end{align*}
Using $a = \lapmv \lapv a$ we arrive at
\[
 \abs{II(x)} \aleq \int \frac{\abs{a(y)-a(x)}\abs{\lapv b(y)}}{\abs{x-y}^{1+s}}\ \d y \aleq \int \int \frac{\abs{\abs{y-z}^{-1+\frac{1}{2}}-\abs{x-z}^{-1+\frac{1}{2}}}\abs{\lapv b(y)} \abs{\lapv a(z)}}{\abs{x-y}^{1+s}} \ \d y\ \d z.
\]
For almost all $x,y,z$ we get from Lemma~\ref{lem:MultiplierEstimate} choosing $\delta:=\frac14+\frac s2$
\[
\frac{\abs{\abs{y-z}^{-1+\frac{1}{2}}-\abs{x-z}^{-1+\frac{1}{2}}}}{\abs{x-y}^{1+s}} \leq 
\brac{\abs{x-z}\abs{x-y}}^{-1+\fracm{4}-\frac{s}{2}}
+ \brac{\abs{y-z}\abs{x-y}}^{-1+\fracm{4}-\frac{s}{2}},
\]
which implies, using again~\eqref{eq:fraclap-}
\[
 \abs{II(x)} \aleq \lapms{\frac{1-2s}{4}} \abs{\lapv a}\ \lapms{\frac{1-2s}{4}} \abs{\lapv b} + \lapms{\frac{1-2s}{4}} \brac{\abs{\lapv b}\ \lapms{\frac{1-2s}{4}} \abs{\lapv a}}.
\]

By symmetry a respective estimate holds also for the term $III$.
Applying Lemma~\ref{lem:LowerOrderProducts.v2} one concludes.
\end{proof}

\begin{proof}[Proof of Lemma~\ref{lem:CriticalTerm}]
 Again the proof relies on quasi-locality. First we decompose
 \begin{align*}
  \left|\int \lapv  g_i' \omega_{ij} H_{1/2}(g'_j,\phi) \right|
  &\aleq  \left|\; \int_{B_{\Lambda^{1/2}r}}  \lapv  g_i' \omega_{ij} H_{1/2}(g'_j,\phi) \right|
  + \sum_{k=1}^\infty \left|\;
  \int_{A^k_{\Lambda^{1/2}, r}} 
  \lapv  g_i' \omega_{ij} H_{1/2}(g'_j,\phi) \right|\\
  &\aleq \|\lapv g'\|_{(2,\infty),B_{\Lambda^{1/2}r}} 
  \|H_{1/2}(g'_j,\phi)\|_{(2,1), B_{\Lambda^{1/2}r}} \\ 
  & \quad +
  \sum_{k=1}^\infty \|\lapv  g'\|_{(2,\infty),A_{\Lambda^{1/2},r}^{k}}
  \|H_{1/2}(g'_j,\phi)\|_{(2,1),A^k_{\Lambda^{1/2},r}}   .
 \end{align*}
  Using Lemma~\ref{lem:localizedHEst},
  the first summand can be estimated by 
 \begin{align*}
 \vrac{\lapms{s}\lapv \phi}_{(\frac{2}{1-2s},1)} \ \brac{
  \|\lapv g'\|_{(2,\infty),B_{\Lambda r}}^2 +  \|\lapv g'\|_{(2,\infty),B_{\Lambda r}}\ \Lambda^{-\theta} \sum_{l=1}^\infty 2^{-\theta l} \|\lapv  g'\|_{(2,\infty),B_{\Lambda 2^lr}}
  }.
 \end{align*}
 Applying the second part of Lemma~\ref{lem:localizedHEst}, the infinite sum can be estimated by
 \[ \Lambda^{-\theta}\| \lapms{s}  \lapv \phi\|_{(\frac 2 {1-2 s},1)}\ \|\lapv g'\|_{(2,\infty)} \ \sum_{k=2}^\infty 
  2^{-\theta k} \|\lapv  g'\|_{(2,\infty),B_{2^k\Lambda r}}.\]
\end{proof}

%
%
%
%

\subsection{Iteration lemma}

In order to prove Dirichlet growth, we need an iteration lemma
whose proof is based on the technique presented in~\cite[\href{http://www.math.ucla.edu/~tao/254a.1.01w/notes4.dvi}{notes4}, p.~11]{TaoLectures}.
The statement is also similar to the corresponding one appearing in~\cite{DR11Sphere}.
One should see this as a generalized version of De Giorgi's Iteration Lemma, cf., e.g., \cite{GiaquintaMI}.

\begin{lemma}[Iteration lemma] \label{lem:IterationLemma}
 Let $C<\infty$ and $\theta >0 $ be given.

 If $b_k\geq 0$, $k \in \mathbb N_0$, satisfy
 \begin{equation} \label{eq:Iteration}
  b_{k+m} \leq \varepsilon b_{k} +
    C\left(2^{-\theta (k+m)} + 2^{-\theta m}\sum_{l=1}^k 2^{-\theta l} b_{k-l}\right)  
 \end{equation}
 for all $k \in \mathbb N_0$ and $\varepsilon>0$ is small enough, $m$ is 
 big enough, then
 \begin{equation*}
  b_k \aleq 2^{-\tilde \theta k} 
 \end{equation*}
 for all $k\in \mathbb N_0$ with $\tilde \theta = \theta /2$.
\end{lemma}

\begin{proof}
 We will prove that
 \begin{equation} \label{eq:sumConverges}
   \sum_{k=0}^\infty 2^{\tilde \theta k} b_k \leq C\sum_{l=0}^m 2^{\tilde \theta l} b_l+ C.
 \end{equation}
 Especially, the infinite sum converges and hence the summands
 are a null series, which proves the lemma.

  Multiplying Equation~\eqref{eq:Iteration} with $2^{\tilde \theta k}$,
  and summing over $k$ we get 
 \begin{align*}
    2^{-\tilde \theta m} \sum_{k=0}^\infty {2}^{\tilde \theta (k+m)} b_{k+m}
     &\leq \varepsilon \sum_{k=0}^\infty 2^{\tilde \theta k} b_{k} 
     + C \sum_{k=0}^\infty 2^{-\theta (k+m)} 2^{\tilde \theta k} 
     + C 2^{-\theta m}\sum_{k=0}^\infty \sum_{l=1}^k 2^{-\theta l} 2^{\tilde \theta k} b_{k-l}\\
     &\leq \varepsilon \sum_{k=0}^\infty  2^{\tilde \theta k} b_{k}
     + C 2^{-\theta m} \sum_{k'=0}^\infty \sum_{l=1}^\infty 2^{-\theta l} 2^{\tilde \theta (k'+l)} b_{k'} 
     + C 2^{-\theta m} \\
     &\leq \varepsilon \sum_{k=0}^\infty 2^{\tilde \theta k} b_{k}
     + \tilde C 2^{-\theta m} \sum_{k'=0}^\infty 2^{\tilde \theta k'} b_{k'}
     + C 2^{-\theta m}.
 \end{align*}
  Assuming that $m$ is so large that
  $\tilde C 2^{-m(\theta  -\tilde \theta)} = \tilde C 2^{-\frac{m\theta}2} < 1/4$ and $0<\varepsilon<\frac 1 {4} 2^{-\theta m}$
  we get
 \begin{align*}
   \sum_{k=0}^\infty 2^{\tilde \theta k} b_k 
  \leq \frac 1 2 \sum_{k=0}^\infty 2^{\tilde \theta k} b_k  + \sum_{l=0}^m  2^{\tilde \theta l} b_l+ C
 \end{align*}
 and hence
 \begin{equation*}
  \sum_{k=0}^\infty 2^{\tilde \theta k} b_k \leq C \sum_{l=0}^m 2^{\tilde \theta l}b_l + C
 \end{equation*}
 if the infinite series converges. 
 If the sum is not known to converge, we apply the above argument to the cut-off series
 \begin{equation*}
  \begin{cases}
   \tilde b_k =b_k &\text{ if } k \le N, \\
   \tilde b_k = 0 &\text{ else}
  \end{cases}
 \end{equation*}
 and get the uniform bound
 $
  \sum_{k=0}^N  2^{\tilde \theta k} b_k \aleq \sum_{l=0}^m 2^{\tilde\theta l} b_l+ 1.
 $
  Letting $N \rightarrow \infty$ we get \eqref{eq:sumConverges}.
\end{proof}

\end{appendix}

\nocite{GrafakosC,GrafakosM}

\end{document}